\documentclass[11pt,reqno]{amsart}
\usepackage{fullpage}
\usepackage{fourier}
\usepackage{mathrsfs,amssymb,graphicx,verbatim,amsmath,amsfonts}
\usepackage{paralist}
\usepackage[breaklinks,pdfstartview=FitH]{hyperref}
\usepackage{upgreek}
\usepackage{dutchcal}
\renewcommand{\le}{\leqslant}
\renewcommand{\ge}{\geqslant}
\renewcommand{\leq}{\leqslant}
\renewcommand{\geq}{\geqslant}
\renewcommand{\setminus}{\smallsetminus}
\renewcommand{\gamma}{\upgamma}

\newcommand{\MM}{\mathcal{M}}

\renewcommand{\setminus}{\smallsetminus}

\newcommand{\n}{\{1,\ldots,n\}}

\newcommand{\f}{\varphi}

\newcommand{\e}{\varepsilon}
\newcommand{\R}{\mathbb R}
\newcommand{\1}{\mathbf 1}

\newtheorem{theorem}{Theorem}[section]
\newtheorem{proposition}[theorem]{Proposition}

\theoremstyle{remark}
\newtheorem{conjecture}[theorem]{Conjecture}

\newtheorem{remark}[theorem]{Remark}
\newtheorem{question}[theorem]{Question}
\theoremstyle{definition}

\renewcommand{\subset}{\subseteq}

\newcommand{\C}{\mathbb C}

\newcommand{\E}{\mathbb{ E}}

\newcommand{\eqdef}{\stackrel{\mathrm{def}}{=}}

\def\E{{\mathbb E}}

\def\R{{\mathbb R}}
\def\C{{\mathbb C}}
\def\Pr{{\mathbb P}}

\def\1{{\mathbf 1}}
\def\ud{{\mathrm d}}

\title{Moments of the distance between independent random vectors}

\author{Assaf Naor and Krzysztof Oleszkiewicz}
\date{}
\thanks{A.N. was supported by the Packard Foundation and the Simons Foundation.  The research that is presented here was conducted under the auspices of the Simons Algorithms and Geometry (A\&G) Think Tank. K.O. was partially supported by the National Science Centre, Poland, project number 2012/05/B/ST1/00412.}
\begin{document}

\maketitle

\begin{abstract}
We derive various sharp bounds on moments of the distance between two independent random vectors taking values in a Banach space.
\end{abstract}
\section{Introduction}

Throughout what follows, all Banach spaces are tacitly assumed to be separable. This  assumption removes the need to discuss  measurability side-issues; alternatively one could consider throughout only the special case of finitely-supported random variables, which captures all of the key ideas.  We will also tacitly assume that all Banach spaces are over the complex scalars $\C$. This assumption is convenient for the ensuing proofs, but the main statements (namely, those that do not mention complex scalars explicitly) hold over the real scalars as well, through a standard complexification procedure. All the notation and terminology from Banach space theory that occurs below is basic and standard, as in e.g.~\cite{LT77}.

Our starting point is the following question. What is the smallest $C>0$ such that for every Banach space $(F,\|\cdot\|_{\!F})$ and every  two independent $F$-valued integrable random vectors $X,Y\in L_1(F)$ we have
\begin{equation}\label{eq:3 in intro}
\inf_{z\in F}\E\left[\|X-z\|_{\!F}^{\phantom{p}}+\|Y-z\|_{\!F}^{\phantom{p}}\right]\le C\E\left[\|X-Y\|_{\!F}^{\phantom{p}}\right]?
\end{equation}
We will reason that~\eqref{eq:3 in intro} holds with $C=3$, and that $C=3$ is the sharp constant here. More generally, we have the following theorem.
\begin{theorem} \label{general}
Suppose that $p \geq 1$ and $(F, \| \cdot\|_{\!F})$ is a Banach space.  Let $X,Y\in L_p(F)$ be two
independent $F$-valued $p$-integrable random vectors. Then
\begin{equation}\label{eq:p version}
\inf_{z \in F} \E\left[\| X-z\|_{\!F}^{p}+\| Y-z\|_{\! F}^{p} \right] \leq \frac{3^p}{2^{p-1}} \E\left[\| X-Y\|_{\!F}^{p}\right].
\end{equation}
The constant $\frac{3^p}{2^{p-1}}$ in~\eqref{eq:p version} cannot be improved.
\end{theorem}
The Banach space $F$ that exhibits this sharpness of~\eqref{eq:p version}   is, of course, a subspace of $\ell_\infty$, but we do not know what is the optimal constant in~\eqref{eq:p version} when $F=\ell_\infty$ itself. More generally, understanding the meaning of the optimal constant in~\eqref{eq:p version} for specific Banach spaces is an interesting question, which we investigate in the rest of the present work for certain special classes of Banach spaces but do not fully resolve.

\subsection{Geometric motivation}Our interest in~\eqref{eq:3 in intro} arose from investigations of~\cite{ANN18} in the context of Riemannian/Alexandrov geometry.  It is well established throughout an extensive geometric literature that a range of useful quadratic distance inequalities for a metric space $(\MM,d_\MM)$ arise if one imposes bounds on its curvature in the sense of Alexandrov. The term ``quadratic'' here indicates that these inequalities involve squares of distances between finite point configurations in $\MM$. A phenomenon that was established in~\cite{ANN18} is that any such quadratic metric inequality that holds for every Alexandrov space of nonnegative curvature becomes valid in any metric space whatsoever if one removes the squaring of the distances, i.e., in essence upon ``linearization'' of the inequality; see~\cite{ANN18} for a precise formulation. This led naturally to the question whether the same phenomenon holds for Hadamard spaces (complete simply connected spaces whose Alexandrov curvature in nonpositive); see~\cite{ANN18} for an extensive discussion as well as the recent negative resolution of this question in~\cite{EMN18}. In the context of a Hadamard space $(\MM,d_\MM)$, the analogue of~\eqref{eq:3 in intro} is that independent finitely-supported $\MM$-valued random variables $X,Y$ satisfy
\begin{equation}\label{eq:Hadamard}
\inf_{z\in \MM}\E\left[d_\MM(X,z)^2+d_\MM(Y,z)^2\right]\le \E\left[d_\MM(X,Y)^2\right].
\end{equation}
See~\cite{ANN18} for a standard derivation of~\eqref{eq:Hadamard}, where $z\in \MM$ is an appropriate ``geometric barycenter,'' namely it is obtained as the minimizer of the expected squared  distance from $X$ to $z$. As explained in~\cite{ANN18}, by using~\eqref{eq:Hadamard} iteratively one can obtain quadratic metric inequalities that hold in any Hadamard space and serve as obstructions for certain geometric embeddings.  The ``linearized'' version of~\eqref{eq:Hadamard}, in the case of Banach spaces and allowing for a loss of a factor $C$, is precisely~\eqref{eq:3 in intro}. So, in the spirit of~\cite{ANN18} it is natural to ask what is the smallest $C$ for which it holds. This is what we address here, leading to analytic questions about Banach spaces that are interesting in their own right from the probabilistic and geometric perspective.  We note that there are questions along these lines that~\cite{ANN18} raises and remain open; see e.g.~\cite[Question~32]{ANN18}.

\subsection{Probabilistic discussion} The inequality which reverses~\eqref{eq:3 in intro} holds trivially as a consequence of the triangle inequality, even when $X$ and $Y$ are not necessarily independent.  Namely, any $X,Y\in L_1(F)$ satisfy
\begin{equation*}
\E\left[\|X-Y\|_{\!F}^{\phantom{p}}\right]\le \inf_{z\in F}\E\left[\|X-z\|_{\!F}^{\phantom{p}}+\|Y-z\|_{\!F}^{\phantom{p}}\right].
\end{equation*}
So, the above discussion is about the extent to which this use of the triangle inequality can be reversed.

Since the upper bound that we seek is in terms of the distance in $L_p(F)$ between independent copies of $X$ and $Y$, this can be further used to control from above expressions such as
$\E[\| X-Y\|_{\!F}^{p}]$ for $X$ and $Y$ not necessarily independent in terms of $\E[\| X'-Y'\|_{\!F}^{p}]$, where $X'$ and $Y'$ are independent, $X'$ has the same distribution as $X$, and $Y'$ has the same distribution as $Y$.

In order to analyse the inequality~\eqref{eq:p version} in a specific Banach space $(F,\|\cdot\|_F)$, we consider the following geometric moduli.  Given $p\ge 1$ let $\boldsymbol{\mathcal{b}}_p(F,\|\cdot\|_F)$, or simply $\boldsymbol{\mathcal{b}}_p(F)$ if the norm is clear from the context, be the infimum over those $\boldsymbol{\mathcal{m}}>0$ such that every  independent $F$-valued random variables $X,Y\in L_p(F)$ satisfy
\begin{equation}\label{eq:def b}
\inf_{z \in F} \E\left[\| X-z\|_{\!F}^{p}+\E\| Y-z\|_{\! F}^{p} \right] \leq \boldsymbol{\boldsymbol{\mathcal{b}}} \E\left[\| X-Y\|_{\!F}^{p}\right].
\end{equation}
Thus, $\boldsymbol{\mathcal{b}}_p(F)$ is precisely the best possible constant in the $L_p(F)$-analogue of the aforementioned  barycentric inequality~\eqref{eq:Hadamard}. The use of the letter ``$\boldsymbol{\mathcal{b}}$'' in this notation is in reference to the word ``barycentric.'' Theorem~\ref{general} asserts that $\boldsymbol{\mathcal{b}}_p(F)\le 3^p/2^{p-1}$, and that this bound cannot be improved in general.

Let  $\boldsymbol{\mathcal{m}}_p(F,\|\cdot\|_F)> 0$, or simply $\boldsymbol{\mathcal{m}}_p(F)$ if the norm is clear from the context, be the infimum over those $\boldsymbol{\mathcal{b}}>0$ such that every  independent $F$-valued random variables $X,Y\in L_p(F)$ satisfy
\begin{equation}\label{eq:def m}
\E\left[\bigg\| X-\frac12\E[X]-\frac12\E[Y]\bigg\|_{\!F}^{p}+\bigg\| Y-\frac12\E[X]-\frac12\E[Y]\bigg\|_{\!F}^{p} \right]\le \boldsymbol{\mathcal{m}} \E\left[\| X-Y\|_{\!F}^{p}\right].
\end{equation}
The use of the letter ``$\boldsymbol{\mathcal{m}}$'' in this notation is in reference to the word ``mixture,'' since the left-hand side of~\eqref{eq:def m} is equal to $2\E[\|Z-\E[Z]\|_F^p]$, where $Z\in L_p(F)$ distributed according to the mixture of the laws of $X$ and $Y$, namely $X$ is the $F$-valued random vector such that for every Borel set
$A \subset F$,
\begin{equation}\label{eq:def mixture}
\Pr[Z \in A]=\frac12\Pr[X \in A]+\frac12\Pr[Y \in A]
\end{equation}
Obviously $\boldsymbol{\mathcal{b}}_p(F)\le \boldsymbol{\mathcal{m}}_p(F)$, because~\eqref{eq:def m} corresponds to choosing $z= \frac12\E[X]+\frac12\E[Y]\in F$ in~\eqref{eq:def b}.

While we sometimes bound $\boldsymbol{\mathcal{m}}_p(F)$ directly, it is beneficial to refine the considerations through the study of two further moduli that are natural in their own right and, as we shall see later, their use can lead to better bounds. Firstly, let $\boldsymbol{\mathcal{r}}_p(F,\|\cdot\|_F)$, or simply $\boldsymbol{\mathcal{r}}_p(F)$ if the norm is clear from the context, be the infimum over those $\boldsymbol{\mathcal{r}}>0$ such that every  independent $F$-valued random variables $X,Y\in L_p(F)$ satisfy
\begin{equation}\label{eq:roundness}
 \E\left[\|X-X'\|_{\!F}^p\right]+ \E\left[\|Y-Y'\|_{\!F}^p\right]\le \boldsymbol{\mathcal{r}} \E\left[\|X-Y\|_{\!F}^p\right],
\end{equation}
where $X',Y'$ are independent copies of $X$ and $Y'$, respectively. The use of the letter ``$\boldsymbol{\mathcal{r}}$'' in this notation is in reference to the word ``roundness,'' as we shall next explain.

Observe also that~\eqref{eq:roundness} is a purely metric condition, i.e., it involves only distances between points. So, it makes sense to investigate~\eqref{eq:roundness} in any metric space $(\MM,d_\MM)$, namely to study the inequality
\begin{equation}\label{eq:roundness-metric}
 \E\left[d_\MM(X,X')^p\right]+ \E\left[d_\MM(Y,Y')^p\right]\le \boldsymbol{\boldsymbol{\mathcal{r}}} \E\left[d_\MM(X,Y)^p\right].
\end{equation}
One requires~\eqref{eq:roundness-metric} to hold for $\MM$-valued independent  random variables $X,X',Y,Y'$ (say, finitely-supported, to avoid measurability assumptions) such that each of the pairs  $X,X'$ and $Y,Y'$ is identically distributed.

To the best of our knowledge, condition~\eqref{eq:roundness-metric} was first studied systematically by Enflo~\cite{Enf69}, who defined a metric space $(\MM,d_\MM)$ to have {\em generalized roundness} $p$ it it satisfies~\eqref{eq:roundness-metric} with $\boldsymbol{\mathcal{r}}=2$. He proved that $L_p$ has generalized roundness $p$ for $p\in [1,2]$, and ingeniously used this notion to answer an old question of Smirnov. See~\cite{DGLY02} for a relatively recent example of substantial  impact of Enflo's approach. By combining~\cite{LTW97} with~\cite{Sch38}, a metric space $(\MM,d_\MM)$ has generalized roundness $p$ if and only if $(\MM,d_\MM^{p/2})$  embeds isometrically into a Hilbert space. The case $\boldsymbol{\mathcal{r}}>1$ of~\eqref{eq:roundness-metric} arose in~\cite{BLMN05} in the context of metric embeddings.

The final geometric modulus that we consider here is a quantity  $\boldsymbol{\mathcal{j}}_p(F,\|\cdot\|_F)$, or simply $\boldsymbol{\mathcal{j}}_p(F)$ if the norm is clear from the context, that is defined to be  the infimum over those $\boldsymbol{\mathcal{j}}\ge 1$ such that every  independent and identically distributed $F$-valued random variables $Z,Z'\in L_p(F)$ satisfy
\begin{equation}\label{eq:variance}
 \boldsymbol{\mathcal{j}}\E\left[\|Z-\E[Z]\|_{\!F}^p\right]\le  \E\left[\|Z-Z'\|_{\!F}^p\right],
\end{equation}
Note that~\eqref{eq:variance} holds with $\boldsymbol{\mathcal{j}} =1$ by Jensen's inequality, so we are asking here for an improvement of (this use of) Jensen's inequality by a definite factor; the letter ``$\boldsymbol{\mathcal{j}}$'' in this notation is in reference to ``Jensen.''

We have the following general bounds, which hold for every Banach space $(F,\|\cdot\|_F)$ and every $p\ge 1$.
\begin{equation}\label{eq:general moduli relations}
\boldsymbol{\mathcal{b}}_p(F)\le \boldsymbol{\mathcal{m}}_p(F)\le \frac{2+\boldsymbol{\mathcal{r}}_p(F)}{2\boldsymbol{\mathcal{j}}_p(F)}.
\end{equation}
Indeed, we already observed the first inequality in~\eqref{eq:general moduli relations}, and the second inequality in~\eqref{eq:general moduli relations} is justified by taking independent random variables $X,Y\in L_p(F)$, considering their  mixture $Z\in L_p(F)$ as defined in~\eqref{eq:def mixture}, letting $X',Y',Z'$ be independent copies of $X,Y,Z$, respectively, and proceeding as follows.
\begin{multline*}
\E\left[\bigg\| X-\frac12\E[X]-\frac12\E[Y]\bigg\|_{\!F}^{p}+\bigg\| Y-\frac12\E[X]-\frac12\E[Y]\bigg\|_{\!F}^{p} \right]\stackrel{\eqref{eq:def mixture}}{=}2\E\left[\|Z-\E[Z]\|_{\!F}^p\right]\stackrel{\eqref{eq:variance}}{\le} \frac{2}{\boldsymbol{\mathcal{j}}_p(F)}\E\left[\|Z-Z'\|_{\!F}^p\right]\\
\stackrel{\eqref{eq:def mixture}}{=}\frac{2}{\boldsymbol{\mathcal{j}}_p(F)}\left(\frac12 \E\left[\|X-Y\|_{\!F}^p\right]+\frac14  \E\left[\|X-X'\|_{\!F}^p\right]+ \frac14\E\left[\|Y-Y'\|_{\!F}^p\right]\right)\le \frac{2}{\boldsymbol{\mathcal{j}}_p(F)}\left(\frac12 +\frac14\boldsymbol{\mathcal{r}}_p(F)\right)\E\left[\|X-Y\|_{\!F}^p\right].
\end{multline*}
Recalling the definition~\eqref{eq:def m} of $\boldsymbol{\mathcal{m}}_p(F)$, this implies~\eqref{eq:general moduli relations}.

Here we prove the following bounds on $\boldsymbol{\mathcal{b}}_p(L_q),\boldsymbol{\mathcal{m}}_p(L_q),\boldsymbol{\mathcal{r}}_p(L_q),\boldsymbol{\mathcal{j}}_p(L_q)$ for $p,q\in [1,\infty)$.

\begin{theorem}\label{thm:Lp case} For every $p,q\in [1,\infty)$ we have $\boldsymbol{\mathcal{j}}_p(L_q)= 2^{c(p,q)}$, where
\begin{equation}\label{eq:df cpq}
c(p,q)\eqdef \min\left\{1,p-1,\frac{p}{q},\frac{p(q-1)}{q}\right\}= \left\{\begin{array}{ll} p-1& \mathrm{if}\ 1\le p\le q\le 2\ \mathrm{or}\ 1\le p\le \frac{q}{q-1}\le 2,\\
\frac{p(q-1)}{q} &\mathrm{if}\ q\le p\le \frac{q}{q-1},\\
\frac{p}{q} &\mathrm{if}\ \frac{q}{q-1}\le p\le q,\\
1& \mathrm{if}\ p\ge \frac{q}{q-1}\ge 2\ \mathrm{or}\ p\ge q\ge 2.\end{array}\right.
\end{equation}
We also have $\boldsymbol{\mathcal{r}}_p(L_q)\le 2^{C(p,q)}$, where
\begin{equation}\label{eq:df Cpq}
C(p,q)\eqdef \left\{\begin{array}{ll} p-1& \mathrm{if}\ \frac{p}{p-1}\le q\le p,\\
\frac{p(q-2)}{q}+1 &\mathrm{if}\  \frac{q}{q-1}\le p\le q,\\
2-\frac{p}{q} &\mathrm{if}\ q\ge 2\  \mathrm{and}\  1\le p\le \frac{q}{q-1},\\
\frac{p}{q} &\mathrm{if}\ q\le 2\  \mathrm{and}\  q\le p\le \frac{q}{q-1},\\
1& \mathrm{if}\ 1\le p\le q\le 2.\end{array}\right.
\end{equation}
In fact, if $\frac{p}{p-1}\le q\le p$, then $\boldsymbol{\mathcal{r}}_p(L_q)=2^{p-1}$, if  $\frac{q}{q-1}\le p\le q$, then $\boldsymbol{\mathcal{r}}_p(L_q)=2^{\frac{p(q-2)}{q}+1 }$, and $\boldsymbol{\mathcal{r}}_p(L_q)=2$ if $1\le p\le q\le 2$. Namely, the above bound on $\boldsymbol{\mathcal{r}}_p(L_q)$ is sharp  in the first, second and fifth ranges in~\eqref{eq:df Cpq}.

Furthermore, $\boldsymbol{\mathcal{b}}_p(L_q)=\boldsymbol{\mathcal{m}}_p(L_q)=2^{2-p}$ if $p\le q\le 2$. More generally, we have the bound
\begin{equation}\label{eq:bm bound}
\boldsymbol{\mathcal{b}}_p(L_q)\le\boldsymbol{\mathcal{m}}_p(L_q)\le \min\left\{\frac{3^p}{2^{p-1}}\left(\frac{\sqrt{2}}{3}\right)^{2c(p,q)},\frac{2^{C(p,q)}+2}{2^{c(p,q)+1}}\right\}.
\end{equation}
\end{theorem}

The upper bound on $\boldsymbol{\mathcal{b}}_p(L_q)$ in~\eqref{eq:bm bound} improves over~\eqref{eq:p version} when $F=L_q$ for all values of $p,q\in [1,\infty)$.  It would be interesting to find the exact value of $\boldsymbol{\mathcal{b}}_p(L_q)$ in the entire range $p,q\in [1,\infty)$. Note that the second quantity in the minimum in the right hand side of~\eqref{eq:bm bound} corresponds to using~\eqref{eq:general moduli relations} together with the bounds on $\boldsymbol{\mathcal{j}}_p(L_q)$ and $\boldsymbol{\mathcal{r}}_p(L_q)$ that Theorem~\ref{thm:Lp case}  provides; when, say, $p=q$, this quantity is smaller than the first quantity in the minimum in the right hand side of~\eqref{eq:bm bound} if and only if $1\le p< 3$.

Theorem\ref{thm:Lp case} states that the constant $C(p,q)$  is sharp in the first, second and fifth ranges in~\eqref{eq:df Cpq}. The following conjecture formulates what we expect to be the sharp values of $\boldsymbol{\mathcal{r}}_p(L_q)$ for all $p,q\in [1,\infty)$.

\begin{conjecture}\label{conj:opt} For all $p,q\in [1,\infty)$ we have $\boldsymbol{\mathcal{r}}_p(L_q)= 2^{C_{\mathrm{opt}}(p,q)}$, where
\begin{equation}\label{eq:df cpq opt}
C_{\mathrm{opt}}(p,q)\eqdef \max\left\{1,p-1,\frac{p(q-2)}{q}+1\right\}=\left\{\begin{array}{ll} p-1& \mathrm{if}\ p\ge 2\  \mathrm{and}\ 1\le q\le p,\\
\frac{p(q-2)}{q}+1 &\mathrm{if}\ q\ge 2\  \mathrm{and}\ 1\le p\le q,\\
1& \mathrm{if}\ p,q\in [1,2].\end{array}\right.
\end{equation}
\end{conjecture}
We will prove later that $\boldsymbol{\mathcal{r}}_p(L_q)\ge  2^{C_{\mathrm{opt}}(p,q)}$, so Conjecture~\eqref{conj:opt} is about improving our upper bounds on $\boldsymbol{\mathcal{r}}_p(L_q)$ in the remaining third and fourth ranges that appear in~\eqref{eq:df Cpq}.

\begin{question} Below we will obtain improvements over~\eqref{eq:p version} for other spaces besides $\{L_q:\ q\in [1,\infty)\}$, including e.g.~the Schatten--von Neumann trace classes (see e.g.~\cite{Sim79}) $\{\mathsf{S}_q:\ q\in (1,\infty)\}$. However, parts of Theorem~\ref{thm:Lp case} rely on ``commutative'' properties of $L_q$ which are not valid for $\mathsf{S}_q$, thus leading to even better bounds in the commutative setting. It would be especially  interesting to obtain sharp bounds in noncommutative probabilistic inequalities such as the roundness inequality~\eqref{eq:roundness} when $F=\mathsf{S}_q$. In particular, we ask what is the value of $\boldsymbol{\mathcal{r}}_1(\mathsf{S}_1)$? At present, we know (as was already shown by Enflo~\cite{Enf69}) that $\boldsymbol{\mathcal{r}}_1(L_1)=2$ while the only bound that we have for $\mathsf{S}_1$  is $\boldsymbol{\mathcal{r}}_1(\mathsf{S}_1)\le 4$.  Note that $4$ is a trivial upper bound here, which holds for every Banach space. Interestingly, it follows from~\cite{BRS17} that $\boldsymbol{\mathcal{r}}_1(\mathsf{S}_1)\ge 2\sqrt{2}$, as explained in Remark~\ref{rem:schatten} below. So, there is a genuine difference between the commutative and noncommutative settings of  $L_1$ and $\mathsf{S}_1$, respectively. As a more modest question, is $\boldsymbol{\mathcal{r}}_1(\mathsf{S}_1)$ strictly less than $4$?

\end{question}

\subsection{Complex interpolation}
We will use basic terminology, notation and results of  complex interpolation of Banach spaces; the relevant background appears in~\cite{Cal64,BL76}. Theorem~\ref{thm:Lp case} is a special case of the following more general result about interpolation spaces. As such, it applies also to random variables that take values in certain spaces other than $L_q$, including, for examples, Schatten--von Neumann trace classes (see e.g.~\cite{Sim79}) and, by an extrapolation theorem of Pisier\cite{Pis79}, Banach lattices of nontrivial type.

\begin{theorem}\label{cor:differences} Fix $\theta\in [0,1]$ and $\frac{2}{2-\theta}\le p\le \frac{2}{\theta}$. Let $(F,\|\cdot\|_{\!F}), (H,\|\cdot\|_{\!H})$ be a compatible pair of Banach spaces such that $(H,\|\cdot\|_{\!H})$ is a Hilbert space.  Then the following estimates hold true.
\begin{equation}\label{eq:rj interpolation}
\boldsymbol{\mathcal{r}}_p([F,H]_\theta)\le 2^{1+(1-\theta)p}\qquad\mathrm{and}\qquad  \mathcal{j}_p([F,H]_\theta)\ge 2^{\frac{\theta p}{2}}.
\end{equation}
Additionally, we have
\begin{equation}\label{eq:mb interpolation}
\boldsymbol{\mathcal{b}}_p([F,H]_\theta)\le \boldsymbol{\mathcal{m}}_p([F,H]_\theta)\le \min\left\{\frac{3^p}{2^{p-1}}\left(\frac{\sqrt{2}}{3}\right)^{p\theta},\frac{1+2^{(1-\theta)p}}{2^{\frac{\theta p}{2}}}\right\}=\left\{\begin{array}{ll}\frac{3^p}{2^{p-1}}\left(\frac{\sqrt{2}}{3}\right)^{p\theta}&\mathrm{if\ } \frac{1}{1-\theta}\le p\le \frac{2}{\theta},\\ \frac{1+2^{(1-\theta)p}}{2^{\frac{\theta p}{2}}}&\mathrm{if\ } \frac{2}{2-\theta}\le p\le \frac{1}{1-\theta}.
\end{array}\right.
\end{equation}
(Note that if the first range of values of $p$ in the right hand side of~\eqref{eq:mb interpolation} is nonempty, then necessarily $\theta\le \frac23$.)
\end{theorem}

The deduction of Theorem~\ref{thm:Lp case} from Theorem~\ref{cor:differences} appears in Section~\ref{sec:interpolation proof} below; in most cases this deduction is nothing more than a direct substitution into Theorem~\ref{cor:differences}, but in some cases a further argument is needed. Theorem~\ref{cor:differences} itself is a special case of the following theorem.

\begin{theorem}\label{thm:interpolation} Fix $\theta\in [0,1]$ and $p\in [1,\infty]$ that satisfy  $\frac{2}{2-\theta}\le p\le \frac{2}{\theta}$. Let $(F,\|\cdot\|_{\!F}), (H,\|\cdot\|_{\!H})$ be a compatible pair of Banach spaces such that $(H,\|\cdot\|_{\!H})$ is a Hilbert space. Suppose that $(\mathcal{X},\mu)$ and $(\mathcal{Y},\nu)$ are probability spaces. Then, for every $f\in L_p(\mu\times \nu;[F,H]_\theta)$  we have
\begin{align}\label{eq:two measures}
\begin{split}
2^{1+(1-\theta)p}\iint_{\mathcal{X}\times \mathcal{Y}}\|f(x,y)\|_{\![F,H]_\theta}^p\ud\mu(x)&\ud\nu(y)\ge \iint_{\mathcal{X}\times \mathcal{X}}\bigg\|\int_\mathcal{Y} \big(f(x,y)-f(\chi,y)\big)\ud\nu(y)\bigg\|_{\![F,H]_\theta}^p\ud\mu(x)\ud\mu(\chi)
\\& \qquad\quad\ \ +\iint_{\mathcal{Y}\times \mathcal{Y}}\bigg\|\int_\mathcal{X} \big(f(x,y)-f(x,\upupsilon)\big)\ud\mu(x)\bigg\|_{\![F,H]_\theta}^p\ud\nu(y)\ud\nu(\upupsilon),
\end{split}
\end{align}
and
\begin{align}\label{eq:two measures-2}
\begin{split}
\frac{3^p}{2^{p-1}}&\left(\frac{\sqrt{2}}{3}\right)^{p\theta}\iint_{\mathcal{X}\times \mathcal{Y}}\|f(x,y)\|_{\![F,H]_\theta}^p\ud\mu(x)\ud\nu(y)\\&\ge
\int_{\mathcal{X}}\bigg\|\int_\mathcal{Y} f(x,y)\ud\nu(y)-\frac12\iint_{\mathcal{X}\times\mathcal{Y}}f(\chi,\upupsilon)\ud\mu(\chi)\ud\nu(\upupsilon)\bigg\|_{\![F,H]_\theta}^p\ud\mu(x)
\\& \qquad+\int_{\mathcal{Y}}\bigg\|\int_\mathcal{X} f(x,y)\ud\mu(x)-\frac12\iint_{\mathcal{X}\times\mathcal{Y}}f(\chi,\upupsilon)\ud\mu(\chi)\ud\nu(\upupsilon)\bigg\|_{\![F,H]_\theta}^p
\ud\nu(y).
\end{split}
\end{align}
Furthermore, if $g\in L_p(\mu\times \mu;[F,H]_\theta)$, then
\begin{equation}\label{eq:one measure}
2^{\left(1-\frac{\theta}{2}\right)p}\iint_{\mathcal{X}\times\mathcal{X}}\|g(x,\chi)\|_{\![F,H]_\theta}^p\ud\mu(x)\ud\mu(\chi)\ge \int_{\mathcal{X}} \bigg\|\int_{\mathcal{X}} \big(g(x,\chi)-g(\chi,x)\big)\ud\mu(x)\bigg\|_{\![F,H]_\theta}^p\ud\mu(\chi).
\end{equation}
\end{theorem}
\begin{proof}[Proof of Theorem~\ref{cor:differences} assuming Theorem~\ref{thm:interpolation}] Let $X$ and  $Y$ be independent $p$-integrable $[F,H]_\theta$-valued random vectors. Due to the independence assumption, without loss of generality there are probability spaces $(\mathcal{X},\mu)$ and $(\mathcal{Y},\nu)$ such that $X$ and $Y$ are elements of  $L_p(\mu\times\nu;[F,H]_\theta)$ that depend only on the first variable and second variable, respectively. Then~\eqref{eq:two measures} and~\eqref{eq:two measures-2} applied to $f=X-Y$ become
$$
\E\left[\left\|X-X'\right\|^p_{\![F,H]_\theta}\right]+\E\left[\left\|Y-Y'\right\|^p_{\![F,H]_\theta}\right]\le 2^{1+(1-\theta)p}\E\left[\left\|X-Y\right\|^p_{\![F,H]_\theta}\right],
$$
and
$$
\E\left[\bigg\|X-\frac12\E[X]-\frac12\E[Y]\bigg\|^p_{\![F,H]_\theta}\right]+\E\left[\bigg\|Y-\frac12\E[X]-\frac12\E[Y]\bigg\|^p_{\![F,H]_\theta}\right]
\le \frac{3^p}{2^{p-1}}\left(\frac{\sqrt{2}}{3}\right)^{p\theta}\E\left[\left\|X-Y\right\|^p_{\![F,H]_\theta}\right].
$$
We therefore established the first inequality in~\eqref{eq:rj interpolation} as well as the upper bound on $\boldsymbol{\mathcal{m}}_p([F,H]_\theta)$ that corresponds to the first term in the minimum that appears in~\eqref{eq:mb interpolation}.

Similarly, due to the  fact that $X$ and $X'$ are i.i.d., without loss of generality there is a probability space $(\mathcal{X},\mu)$ such that $X$ and $X'$ are elements of  $L_p(\mu\times\mu;[F,H]_\theta)$ that depend only on the first variable and second variable, respectively.  Then, \eqref{eq:one measure} applied to $g=X-X'$ simplifies to give
$$
\E\left[\left\|X-X'\right\|^p_{\![F,H]_\theta}\right]\ge 2^{\frac{\theta p}{2}}\E\left[\left\|X-\E[X]\right\|^p_{\![F,H]_\theta}\right].
$$
This establishes the second inequality in~\eqref{eq:rj interpolation}, as well as the upper bound on $\boldsymbol{\mathcal{m}}_p([F,H]_\theta)$ that corresponds to the second term in the minimum that appears in~\eqref{eq:mb interpolation}, due to~\eqref{eq:general moduli relations}.
\end{proof}

The first and third inequalities of Theorem~\ref{thm:interpolation} are generalizations of results that appeared in the literature. Specifically, \eqref{eq:two measures}  generalizes Lemma~6 of~\cite{BLMN05}, and~\eqref{eq:one measure} generalizes Lemma~5 of~\cite{Nao01}, which is itself inspired   by a step within the proof of Theorem~2 of~\cite{WWH71}. The proof of Theorem~\ref{thm:interpolation}, which appears in Section~\ref{sec:interpolation proof} below, differs  from the proofs of~\cite{WWH71,Nao01,BLMN05}, but relies on the same ideas.

\section{Proof of Theorem~\ref{general}}

Let $(F,\|\cdot\|_F)$ be a Banach space. Fix $p\ge 1$. Theorem~\ref{general} asserts that $\boldsymbol{\mathcal{b}}_p(F)\le \frac{3^p}{2^{p-1}}$. In fact, $\boldsymbol{\mathcal{m}}_p(F)\le \frac{3^p}{2^{p-1}}$, which  is stronger by~\eqref{eq:general moduli relations}. To see this, let $X,Y\in L_p(F)$ be independent random vectors and observe that
\begin{align*}
\begin{split}
\E\left[\bigg\|X-\frac12\E[X]-\frac12\E[Y]\bigg\|_{\!F}^p\right]&=
\frac{3^p}{2^p}\E\left[\bigg\|\frac23\left(X-\E[Y]\right)+\frac13\left(\E[Y]-\E[X]\right)\bigg\|_{\!F}^p\right]\\
&\le \frac{3^p}{2^p}\left(\frac23\E\left[\left\|X-\E[Y]\right\|_{\!F}^p\right]+\frac13\E\left[\left\|\E[Y]-\E[X]\right\|_{\!F}^p\right]\right)\le
\frac{3^p}{2^p}\E\left[\left\|X-Y\right\|_{\!F}^p\right],
\end{split}
\end{align*}
where the penultimate step holds due to the convexity of $\|\cdot\|_{\!F}^p$ and the final step holds because, by Jensen's inequality, both  $\E\left[\left\|X-\E[Y]\right\|_{\!F}^p\right]=\E\left[\left\|\E_Y[X-Y]\right\|_{\!F}^p\right]$ and $\E\left[\left\|\E[Y-X]\right\|_{\!F}^p\right]$ are at most $\E\left[\left\|X-Y\right\|_{\!F}^p\right]$. The symmetric reasoning with $X$ replaced by $Y$ now gives
\begin{multline*}
\E\left[\bigg\|X-\frac12\E[X]-\frac12\E[Y]\bigg\|_{\!F}^p\right]+\E\left[\bigg\|Y-\frac12\E[X]-\frac12\E[Y]\bigg\|_{\!F}^p\right]\\\le 2\max\left\{\E\left[\bigg\|X-\frac12\E[X]-\frac12\E[Y]\bigg\|_{\!F}^p\right],
\E\left[\bigg\|Y-\frac12\E[X]-\frac12\E[Y]\bigg\|_{\!F}^p\right]\right\}\le \frac{3^p}{2^{p-1}}\E\left[\left\|X-Y\right\|_{\!F}^p\right].
\end{multline*}
This shows that $\boldsymbol{\mathcal{m}}_p(F)\le \frac{3^p}{2^{p-1}}$. It remains to prove that the bound $\boldsymbol{\mathcal{b}}_p(F)\le \frac{3^p}{2^{p-1}}$ is optimal for general $F$.

Fix an integer $n\ge 2$ and consider $$
F_{n}\eqdef \bigg\{x \in \C^{2n}: \sum_{k=1}^{2n} x_{k}=0\bigg\},
$$
equipped with
supremum norm inherited from $\ell_{\infty}^{2n}$. We will prove that
\begin{equation}\label{eq:Fn lower}
\boldsymbol{\mathcal{b}}_p(F_n)\ge 2\left(\frac{3}{2}-\frac{1}{n}\right)^p\xrightarrow[n\to \infty]{}  \frac{3^p}{2^{p-1}}.
\end{equation}

Denote by
$\{e_{k}\}_{k=1}^{2n}$ the standard coordinate basis of $\ell_{\infty}^{2n}$. Define two $n$-element sets
 $A_{n}, B_{n} \subset F_{n}$ by
\[
A_{n}\eqdef \bigg\{ (3n-2)e_{j}-(n+2)\sum_{k\in \n\setminus\{j\}}e_{k}+
(n-2)\sum_{k=n+1}^{2n}e_{k}:\, j\in \n\bigg\},
\]
and
\[
B_{n}\eqdef\bigg\{(n-2) \sum_{k=1}^{n}e_{k}+(3n-2)e_{j}-
(n+2)\sum_{k\in \{n+1,\ldots,2n\}\setminus \{j\}}e_{k}:\ j\in \{n+1,\ldots,2n\}\bigg\}.
\]
Note that $A_n$ and $B_n$ are indeed subsets of $F_n$ because $3n-2-(n-1)(n+2)+n(n-2)=0$. Let $X,Y$ be independent and uniformly distributed on
$A_{n},B_{n}$, respectively. One checks that $\| a-b\|_\infty=2n$ for any $a \in A_{n}$
and $b \in B_{n}$. So,  $\E\left[\| X-Y\|_\infty^{p}\right]=(2n)^{p}$. The desired  bound~\eqref{eq:Fn lower} will follow if we demonstrate that
\begin{equation}\label{eq: for all z}
\forall\, z\in F_n,\qquad \E\left[\| X-z\|_\infty^{p}+\| Y-z\|_\infty^{p}\right]\ge 2(3n-2)^p.
\end{equation}

The proof of~\eqref{eq: for all z} proceeds via symmetrization.  For permutations $\sigma, \rho \in S_{n}$,
define $T_{\sigma, \rho}: F_{n} \rightarrow F_{n}$ by
\[
\forall\, x=(x_1,\ldots,x_{2n})\in F_n,\qquad T_{\sigma, \rho}(x)\eqdef \left(x_{\sigma(1)}, x_{\sigma(2)}, \ldots, x_{\sigma(n)},
x_{n+\rho(1)},x_{n+\rho(2)}, \ldots, x_{n+\rho(n)}\right).
\]
$T_{\sigma,\rho}$ is a linear isometry of $F_{n}$
and the sets $A_{n}$ and $B_{n}$ are $T_{\sigma, \rho}$-invariant. Hence, for
any $z \in F_{n}$,
\begin{align}\label{eq:perm}
\begin{split}
\E &\left[\| X-z\|_\infty^{p}\right]=\frac{1}{(n!)^{2}}\sum_{\sigma, \rho \in S_{n}} \E \left[\| T_{\sigma,\rho}(X)-z\|_\infty^{p}\right]=
\frac{1}{(n!)^2}\sum_{\sigma, \rho \in S_{n}} \E \left[\| X-T_{\sigma^{-1}, \rho^{-1}}(z)\|_\infty^{p}\right]\\
&\geq
\E \left[\Big\| X-\frac{1}{(n!)^{2}}\sum_{\sigma, \rho \in S_{n}} T_{\sigma^{-1}, \rho^{-1}}(z) \Big\|_\infty^{p}\right]
=
\E \left[\Big\| X-\frac{z_1+\ldots+z_n}{n}\sum_{k=1}^{n}e_{k}+\frac{z_1+\ldots+z_n}{n}\sum_{k=n+1}^{2n}e_{k}\Big\|_\infty^{p}\right].
\end{split}
\end{align}
Denoting $u=(z_1+\ldots+z_n)/n$, it follows from~\eqref{eq:perm}  that  $\E \left[\| X-z\|_\infty^{p}\right]\ge |3n-2-u|^p$, because one of the first $n$ coordinates of any member of the support of $X$ equals $3n-2$. The same argument with $X$ replaced by $Y$ gives that $\E \left[\| Y-z\|_\infty^{p}\right]\ge |3n-2+u|^p$, because now one of the last $n$ coordinates of any member of the support of $Y$ equals $3n-2$. We conclude with the following application of the convexity of $|\cdot|^p$.
\begin{equation*}
\E\left[\| X-z\|_\infty^{p}+\| Y-z\|_\infty^{p}\right]\ge |3n-2-u|^p+ |3n-2+u|^p\ge 2(3n-2)^p.\tag*{\qed}
\end{equation*}

\begin{remark}\label{rem:qFn}
It is worthwhile to examine what the above argument gives if we take the norm on $F_n$ to be the norm inherited from $\ell_q^{2n}$. One computes that $\|a-b\|_q=(2n)^{1+1/q}$ for every $a\in A_n$ and $b\in B_n$. So,
\begin{equation}\label{eq:qp in example}
\E\left[\|X-Y\|_q^p\right]=(2n)^{\frac{p(q+1)}{q}}.
\end{equation}
Also, it follows from the same reasoning that led to~\eqref{eq:perm} that for every $z\in F_n$,
\begin{equation*}
\E\left[\|X-z\|_q^p\right]\ge \E\left[\Big\|X-u\sum_{k=1}^n e_k+u\sum_{k=n+1}^ne_k\Big\|_q^p\right]= \left(|3n-2-u|^q+(n-1)|n+2+u|^q+n|n-2+u|^q\right)^{\frac{p}{q}},
\end{equation*}
and
$$
\E\left[\|Y-z\|_q^p\right]\ge \E\left[\Big\|Y-u\sum_{k=1}^n e_k+u\sum_{k=n+1}^ne_k\Big\|_q^p\right]= \left(|3n-2+u|^q+(n-1)|n+2-u|^q+n|n-2-u|^q\right)^{\frac{p}{q}}.
$$
Hence, using the convexity of the $p$'th power of the $\ell_q$ norm on $\R^3$, we see that
\begin{equation}\label{eq:qp in example2}
\E\left[\|X-z\|_q^p\right]+\E\left[\|Y-z\|_q^p\right]\ge 2\left((3n-2)^q+(n-1)(n+2)^q+n(n-2)^q\right)^{\frac{p}{q}}.
\end{equation}
By contrasting~\eqref{eq:qp in example} with~\eqref{eq:qp in example2} we conclude that
$$
\boldsymbol{\mathcal{b}}_p(F_n,\|\cdot\|_q)\ge 2\left((3n-2)^q+(n-1)(n+2)^q+n(n-2)^q\right)^{\frac{p}{q}}(2n)^{-\frac{p(q+1)}{q}}.
$$
In particular, if we take $p=q\ge 2$ and  $n=\lceil q\rceil$, then we conclude  that $\boldsymbol{\mathcal{b}}_q(F_n,\|\cdot\|_q)\ge \frac{c}{q}\left(\frac32\right)^q$ for some universal constant $c>0$. So, there is very little potential asymptotic gain (as $q\to \infty$) if we know that the Banach space of Theorem~\ref{general} admits an isometric embedding into $L_q$.
\end{remark}

Above, and in what follows, we stated that a normed space admits an isometric embedding into $L_q$ without specifying whether the embedding is linear or not. Later we will need such embeddings to be linear, so we recall that for any $q\ge 1$, by a classical differentiation argument (see~\cite[Chapter~7]{BL00} for a thorough treatment of such reductions to the linear setting), a normed space embeds isometrically into $L_q$ as a metric space if and only if it admits a linear isometric embedding into $L_q$.

Note that the phenomenon of Remark~\ref{rem:qFn} is special to random variables that have different expectations. Namely, if $\E[X]=\E[Y]$, then by Jensen's inequality the ratio that defines $\boldsymbol{\mathcal{b}}_q(F)$ is at most $2$ rather than the aforementioned exponential growth as $q\to \infty$. The following proposition shows that if $F$ is a subspace of $L_q$ for $q\ge 3$, then when $\E[X]=\E[Y]$ this ratio is at most $1$, which is easily seen to be best possible (consider any nontrivial symmetric random variable $X$, and take $Y$ to be identically $0$).

\begin{proposition} \label{embeddingabove3}
Let $(F, \| \cdot \|_F)$ be a Banach space that admits
an isometric   embedding into $L_q$ for some $q \in [3,\infty)$. Then,
for any pair of independent $F$-valued random vectors $X,Y\in L_q(F)$ with $\E[X]=\E[Y]$,
\begin{equation} \label{ineqabove3}
\inf_{z\in F} \E\big[\|X-z\|_{\! F}^q+\|Y-z\|_{\! F}^q\big] \le \E\big[\|X-\E[X]\|_{\! F}^q+\|Y-\E[X]\|_{\! F}^q\big]  \le \E\left[\|X-Y\|_{\!F}^q\right].
\end{equation}
\end{proposition}

\begin{proof} $L_q$ over $\C$ embeds isometrically into $L_q$ over $\R$ (indeed, complex $L_q$ is, as a real Banach space, the same as $L_q(\ell_2^2)$, so this follows from the fact that Hilbert space is isometric to a subspace of $L_q$). So, in Proposition~\ref{embeddingabove3} we may assume that $F$ embeds isometrically into $L_q$ over $\R$, and therefore by integration/Fubini it suffices to prove~\eqref{ineqabove3} for  real-valued random variables. So, our goal is to show that if $X,Y$ are independent mean-zero real random variables with $\E\left[|X|^{q}\right]$, $\E\left[|Y|^{q}\right]<\infty$, then
\begin{equation}\label{sub-additivity}
\E \left[|X+Y|^{q}\right] \geq \E\left[|X|^{q}\right]+\E\left[|Y|^{q}\right].
\end{equation}
The bound~\eqref{ineqabove3} would then follow by applying~\eqref{sub-additivity} to the mean-zero variables $X-\E[X]$ and $\E[X]-Y$.

Note in passing that the assumption $q\ge 3$ is crucial here, i.e.~\eqref{sub-additivity} fails if $q\in (0,3)\setminus\{2\}$.   Indeed, if $\beta\in (0,\frac12)$ and $\Pr[X=1-\beta]=\Pr[Y=1-\beta]=\beta$ and $\Pr[X=-\beta]=\Pr[Y=-\beta]=1-\beta$, then $\E[X]=\E[Y]=0$ but
\begin{equation}\label{eq:beta formula}
\frac{\E \left[|X+Y|^{q}\right]}{\E\left[|X|^{q}\right]+\E\left[|Y|^{q}\right]}= \frac{\beta^22^q(1-\beta)^q+(1-\beta)^22^q\beta^q+2\beta(1-\beta)(1-2\beta)^q}{2\beta(1-\beta)^q+2(1-\beta)\beta^q}.
\end{equation}
If $q\in (0,2)$, then  the right hand side of~\eqref{eq:beta formula} equals $2^{q-2}<1$ for $\beta=\frac12$. If $q\in (2,3)$, then the right hand side of~\eqref{eq:beta formula} equals $1+(2^{q-1}-q-1)\beta +o(\beta)$, which is less than $1$ for small $\beta$ since $2^{q-1}-q-1<0$ for $q\in (2,3)$.

To prove~\eqref{sub-additivity}, for every $s>0$ and $x\in \R$, denote $\phi_{s}(x)={\mathrm{sign}}(x) \cdot |x|^{s}.$ Observe that
\begin{equation}\label{eq:convexity ineq diff}
\forall\, x,y\in \R,\qquad \alpha(x,y)\eqdef |x+y|^{q}-|x|^{q}-|y|^{q}-q\phi_{q-1}(x)y
-qx\phi_{q-1}(y)\ge 0.
\end{equation}
Once~\eqref{eq:convexity ineq diff} is proved, \eqref{sub-additivity} would follow because
\begin{align*}
0\le \E\left[\alpha(X,Y)\right]&=\E \left[|X+Y|^{q}\right] - \E\left[|X|^{q}\right]-\E\left[|Y|^{q}\right]-q\E[Y]\E\left[\phi_{q-1}(X)\right]-q\E[X]\E\left[\phi_{q-1}(Y)\right]\\
&=\E \left[|X+Y|^{q}\right] - \E\left[|X|^{q}\right]-\E\left[|Y|^{q}\right],
\end{align*}
where the penultimate step uses the independence of $X, Y$ and the last step uses $\E[X]=\E[Y]=0$.

It suffices to prove~\eqref{eq:convexity ineq diff} when $q>3$; the case $q=3$ follows by passing to the limit. Once checks that
$$
\frac{\partial^{3}\alpha}{\partial x^{2}\partial y}=
p(p-1)(p-2)\big( \phi_{p-3}(x+y)-\phi_{p-3}(x)\big)\implies \mathrm{sign}\left(\frac{\partial^{3}\alpha}{\partial x^{2}\partial y}(x,y)\right)=\mathrm{sign}(y),
$$
where the last step holds because $\phi_{p-3}$ is increasing. Hence,  $y\mapsto \frac{\partial^{2} \alpha}{\partial x^{2}}(x,y)$ is decreasing for $y<0$ and increasing for $y>0$.  One checks that  $\frac{\partial^{2} \alpha}{\partial x^{2}}(x,0)=0$ for all
$x\in \R$, so $\frac{\partial^{2} \alpha}{\partial x^{2}}(x,y)\ge 0$.  Thus $x\mapsto \alpha(x,y)$ is convex for every fixed $y\in \R$. But $\alpha(0,y)=\frac{\partial \alpha}{\partial x}(0,y)=0$ for any $y\in \R$, i.e.~the tangent to the graph of $x\mapsto \alpha(x,y)$ at $x=0$ is the $x$-axis. Convexity implies that the graph of $x\mapsto \alpha(x,y)$ lies above the $x$-axis, as required.
\end{proof}

We end this section with the following simpler  metric space counterpart of Theorem~\ref{general}.

\begin{proposition} \label{metric}
Fix $p\ge 1$ and let $X$ and $Y$ be independent finitely supported random variables taking values in a metric space $(\MM,d_\MM)$. Then
\begin{equation}\label{eq:roundness2}
\inf_{z \in \MM} \E\left[d_\MM(X,z)^p+d_\MM(Y,z)^p\right]\leq \left(2^{p}+1\right) \E\left[d_\MM(X,Y)^p\right].
\end{equation}
The constant $2^{p}+1$ in~\eqref{eq:roundness2} is optimal.
\end{proposition}

\begin{proof}
Let $X'$ have the same distribution  as $X$ and be independent of $X$ and $Y$. The point-wise inequality
$$
d_\MM(X,X')^p \leq \left( d_\MM(X,Y)+d_\MM(Y,X')\right)^{p} \leq 2^{p-1} \left(d_\MM(X,Y)^p+d_\MM(X',Y)^p\right)
$$
is a consequence of the triangle inequality and the convexity of $(u>0)\mapsto u^p$. By taking expectations,
we obtain $\E \left[d_\MM(X,X')^p\right] \leq 2^{p} \E\left[d_\MM(X,Y)^p\right]$, so that
$$
\inf_{z \in \MM} \E\left[d_\MM(X,z)^p+d_\MM(Y,z)^p\right]\le \E\left[d_\MM(X,X')^p+d_\MM(Y,X')^p\right]\le \left(2^{p}+1\right) \E\left[d_\MM(X,Y)^p\right].
$$

To see that the constant $2^{p}+1$ is optimal, fix $n\in \mathbb{N}$ and let $\MM$ be the complete bipartite graph $\mathsf{K}_{n,n}$, equipped with its shortest-path metric. Equivalently, $\MM$ can be partitioned into two $n$-point subsets $L,R$, and for distinct $x,y\in \MM$ we have  $d_\MM(x,y)=2$ if $\{x,y\}\subset L$ or $\{x,y\}\subset R$, while $d_\MM(x,y)=1$ otherwise. Let $X$ be uniformly distributed over $L$ and $Y$ be uniformly distributed over $R$. Then $d_\MM(X,Y)=1$ point-wise. If $z\in L$, then $d_\MM(Y,z)=1$ point-wise, while $\Pr\left[d_\MM(X,z)=2\right]=\frac{n-1}{n}$ and $\Pr\left[d_\MM(X,z)=0\right]=\frac{1}{n}$.  Consequently,
$$
\frac{\E\left[d_\MM(X,z)^p+d_\MM(Y,z)^p\right]}{\E\left[d_\MM(X,Y)^p\right]}=\frac{n-1}{n}2^p+1\xrightarrow[n\to \infty]{} 2^p+1.
$$
By symmetry, the same holds if $z\in R$.
\end{proof}

\section{Proof of Theorem~\ref{thm:interpolation} and its consequences}\label{sec:interpolation proof}

Here we prove Theorem~\ref{thm:interpolation} and deduce Theorem~\ref{thm:Lp case}. 

\begin{proof}[Proof of Theorem~\ref{thm:interpolation}] The assumption $\frac{2}{2-\theta}\le p\le \frac{2}{\theta}$ implies that $\frac{1}{p}=\frac{1-\theta}{q}+\frac{\theta}{2}$ for some (unique) $q\in [1,\infty]$. We will fix this value of $q$ for the rest of the proof of Theorem~\ref{thm:interpolation}. All of the desired bounds~\eqref{eq:two measures}, \eqref{eq:two measures-2}, \eqref{eq:one measure} hold true when $\theta=0$, namely for every Banach space $(F,\|\cdot\|_F)$ and every $f\in L_q(\mu\times\nu;F)$ we have
\begin{align}\label{eq:two measures'}
\begin{split}
2^{q+1}&\iint_{\mathcal{X}\times \mathcal{Y}}\|f(x,y)\|_{\!F}^q\ud\mu(x)\ud\nu(y)\\&\ge \iint_{\mathcal{X}\times \mathcal{X}}\bigg\|\int_\mathcal{Y} \big(f(x,y)-f(\chi,y)\big)\ud\nu(y)\bigg\|_{\!F}^q\ud\mu(x)\ud\mu(\chi)
\\& \qquad+\iint_{\mathcal{Y}\times \mathcal{Y}}\bigg\|\int_\mathcal{X} \big(f(x,y)-f(x,\upupsilon)\big)\ud\mu(x)\bigg\|_{\!F}^q\ud\nu(y)\ud\nu(\upupsilon),
\end{split}
\end{align}
and
\begin{align}\label{eq:two measures-2'}
\begin{split}
\frac{3^q}{2^{q-1}}&\iint_{\mathcal{X}\times \mathcal{Y}}\|f(x,y)\|_{\!F}^q\ud\mu(x)\ud\nu(y)\\&\ge
\int_{\mathcal{X}}\bigg\|\int_\mathcal{Y} f(x,y)\ud\nu(y)-\frac12\iint_{\mathcal{X}\times\mathcal{Y}}f(\chi,\upupsilon)\ud\mu(\chi)\ud\nu(\upupsilon)\bigg\|_{\!F}^q\ud\mu(x)
\\& \qquad+\int_{\mathcal{Y}}\bigg\|\int_\mathcal{X} f(x,y)\ud\mu(x)-\frac12\iint_{\mathcal{X}\times\mathcal{Y}}f(\chi,\upupsilon)\ud\mu(\chi)\ud\nu(\upupsilon)\bigg\|_{\!F}^q
\ud\nu(y).
\end{split}
\end{align}
Furthermore, if $g\in L_q(\mu\times \mu;F)$, then
\begin{equation}\label{eq:one measure'}
2^{q}\iint_{\mathcal{X}\times\mathcal{X}}\|g(x,\chi)\|_{\!F}^q\ud\mu(x)\ud\mu(\chi)\ge \int_{\mathcal{X}} \bigg\|\int_{\mathcal{X}} \big(g(x,\chi)-g(\chi,x)\big)\ud\mu(x)\bigg\|_{\!F}^q\ud\mu(\chi).
\end{equation}
Indeed, \eqref{eq:two measures'}, \eqref{eq:two measures-2'}, \eqref{eq:one measure'} are direct consequences of the triangle inequality in $L_q(\mu\times\nu;F)$ and $L_q(\mu\times\mu;F)$ and Jensen's inequality, with the appropriate interpretation when $q=\infty$.

By complex interpolation theory (specifically, by combining~\cite[Theorem 4.1.2]{BL76} and~\cite[Theorem~5.1.2]{BL76}), Theorem~\ref{thm:interpolation} will follow if we prove the  $\theta=1$ case of~\eqref{eq:two measures}, \eqref{eq:two measures-2}, \eqref{eq:one measure}. To this end, as $H$ is a Hilbert space and the inequalities in question are quadratic, it suffices to prove them coordinate-wise (with respect to any othonormal basis of $H$), i.e., it suffices to show that for every ($\C$-valued) $f\in L_2(\mu\times\nu)$ and $g\in L_2(\mu\times\mu)$,
\begin{align}\label{eq:scalar 1}
\begin{split}
2\iint_{\mathcal{X}\times \mathcal{Y}}&|f(x,y)|^2\ud\mu(x)\ud\nu(y)\\&\ge \iint_{\mathcal{X}\times \mathcal{X}}\bigg|\int_\mathcal{Y} \big(f(x,y)-f(\chi,y)\big)\ud\nu(y)\bigg|^2\ud\mu(x)\ud\mu(\chi)
\\& \qquad+\iint_{\mathcal{Y}\times \mathcal{Y}}\bigg|\int_\mathcal{X} \big(f(x,y)-f(x,\upupsilon)\big)\ud\mu(x)\bigg|^2\ud\nu(y)\ud\nu(\upupsilon),
\end{split}
\end{align}
and
\begin{align}\label{eq:scalar 2}
\begin{split}
\iint_{\mathcal{X}\times \mathcal{Y}}&|f(x,y)|^2\ud\mu(x)\ud\nu(y)\\&\ge
\int_{\mathcal{X}}\bigg|\int_\mathcal{Y} f(x,y)\ud\nu(y)-\frac12\iint_{\mathcal{X}\times\mathcal{Y}}f(\chi,\upupsilon)\ud\mu(\chi)\ud\nu(\upupsilon)\bigg|^2\ud\mu(x)
\\& \qquad+\int_{\mathcal{Y}}\bigg|\int_\mathcal{X} f(x,y)\ud\mu(x)-\frac12\iint_{\mathcal{X}\times\mathcal{Y}}f(\chi,\upupsilon)\ud\mu(\chi)\ud\nu(\upupsilon)\bigg|^2
\ud\nu(y),
\end{split}
\end{align}
and
\begin{equation}\label{eq:scalar 3}
2\iint_{\mathcal{X}\times\mathcal{X}}|g(x,\chi)|^2\ud\mu(x)\ud\mu(\chi)\ge \int_{\mathcal{X}} \bigg|\int_{\mathcal{X}} \big(g(x,\chi)-g(\chi,x)\big)\ud\mu(x)\bigg|^2\ud\mu(\chi).
\end{equation}
The following derivation of the quadratic scalar inequalities~\eqref{eq:scalar 1}, \eqref{eq:scalar 2}, \eqref{eq:scalar 3} is an exercise in linear algebra.

Let $\{\f_j\}_{j=0}^\infty\subset L_2(\mu)$ and $\{\psi_k\}_{k=0}^\infty \subset L_2(\nu)$ be any orthonormal bases of $L_2(\mu)$ and $L_2(\nu)$, respectively, for which $\f_0=\1_\mathcal{X}$ and $\psi_0=\1_\mathcal{Y}$. Then $\{\f_j\otimes \psi_k\}_{j,k=0}^\infty, \{\f_j\otimes \f_k\}_{j,k=0}^\infty$ and $\{\psi_j\otimes \psi_k\}_{j,k=0}^\infty$ are orthonormal bases of $L_2(\mu\times\nu), L_2(\mu\times\mu)$ and $L_2(\nu\times\nu)$, respectively, where for $\f\in L_2(\mu)$ and $\psi\in L_2(\nu)$ one defines (as usual) $\f\otimes \psi:\mathcal{X}\times \mathcal{Y}\to \C$ by setting $\f\otimes \psi(x,y)=\f(x)\psi(y)$ for  $(x,y)\in \mathcal{X}\times \mathcal{Y}$. We therefore have the following expansions, in the sense of convergence in $L_2(\mu \times \nu)$ and $L_2(\mu\times\mu)$, respectively.
$$
f=\sum_{j=0}^\infty\sum_{k=0}^\infty \langle \xi,\f_j\otimes \psi_k\rangle_{L_2(\mu\times \nu)} \f_j\otimes \psi_k\qquad\mathrm{and}\qquad g=\sum_{j=0}^\infty\sum_{k=0}^\infty \langle \zeta,\f_j\otimes \psi_k\rangle_{L_2(\mu\times \nu)} \f_j\otimes \psi_k.
$$
In particular, by Parseval we have
\begin{equation}\label{eq:parseval use fg}
\|f\|_{L_2(\mu\times \nu)}^2= \sum_{j=0}^\infty\sum_{k=0}^\infty \left|\langle f,\f_j\otimes \psi_k\rangle_{L_2(\mu\times \nu)}\right|^2\qquad\mathrm{and}\qquad \|g\|_{L_2(\mu\times \mu)}^2= \sum_{j=0}^\infty\sum_{k=0}^\infty \left|\langle g,\f_j\otimes \f_k\rangle_{L_2(\mu\times \mu)}\right|^2.
\end{equation}

Define $R_{\mathcal{X}}f\in L_2(\mu\times \mu)$ by
$$
R_\mathcal{X}f\eqdef \sum_{j=1}^\infty \langle f,\f_j\otimes \psi_0\rangle_{L_2(\mu\times \nu)}\f_j\otimes \f_0-\sum_{j=1}^\infty \langle f,\f_j\otimes \psi_0\rangle_{L_2(\mu\times \nu)}\f_0\otimes \f_j.
$$
So, $(\mu\times \mu)$-almost surely $R_{\mathcal{X}}f(x,\chi)=\int_\mathcal{Y} \big(f(x,y)-f(\chi,y)\big)\ud\nu(y)$. Also, define $R_{\mathcal{Y}}f\in L_2(\nu\times \nu)$ by
$$
R_\mathcal{Y}f\eqdef \sum_{j=1}^\infty \langle f,\f_0\otimes \psi_j\rangle_{L_2(\mu\times \nu)}\psi_j\otimes \psi_0-\sum_{j=1}^\infty \langle f,\f_0\otimes \psi_j\rangle_{L_2(\mu\times \nu)}\psi_0\otimes \psi_j.
$$
So,  $(\nu\times \nu)$-almost surely $R_{\mathcal{Y}}f(y,\upupsilon)=\int_\mathcal{X} \big(f(x,y)-f(x,\upupsilon)\big)\ud\nu(x)$. By Parseval in $L_2(\mu\times\mu), L_2(\nu\times\nu), L_2(\mu\times\nu)$,
$$
\left\|R_\mathcal{X}f\right\|_{L_2(\mu\times\mu)}^2+\left\|R_\mathcal{Y}f\right\|_{L_2(\nu\times\nu)}^2=2\sum_{j=1}^\infty\left(\left|\langle f,\f_j\otimes \psi_0\rangle_{L_2(\mu\times \nu)}\right|^2+\left|\langle f,\f_0\otimes \psi_j\rangle_{L_2(\mu\times \nu)}\right|^2\right)\stackrel{\eqref{eq:parseval use fg}}{\le} 2\|f\|_{L_2(\mu\times \nu)}^2.
$$
This is precisely~\eqref{eq:scalar 1}.

Next, for every $\alpha,\beta\in \C$ define $S^\alpha_{\mathcal{X}}f\in L_2(\mu)$ and $S^\beta_{\mathcal{Y}}f\in L_2(\nu)$ by
$$
S^\alpha_\mathcal{X}f\eqdef (1-\alpha)\langle f,\f_0\otimes \psi_0\rangle_{L_2(\mu\times \nu)}\f_0+\sum_{j=1}^\infty \langle f,\f_j\otimes \psi_0\rangle_{L_2(\mu\times \nu)}\f_j,
$$
and
$$
S^\beta_\mathcal{Y}f\eqdef (1-\beta)\langle f,\f_0\otimes \psi_0\rangle_{L_2(\mu\times \nu)}\f_0+\sum_{j=1}^\infty \langle f,\f_0\otimes \psi_j\rangle_{L_2(\mu\times \nu)}\psi_j.
$$
In other words, we have the following identities $\mu$-almost surely and $\nu$-almost surely, respectively.
$$
S^\alpha_\mathcal{X}f(x)=\int_{\mathcal{Y}} f(x,y)\ud\nu(y)-\alpha\iint_{\mathcal{X}\times\mathcal{Y}}f(\chi,\upupsilon)\ud\mu(\chi)\ud\nu(\upupsilon),
$$
and
$$
S^\beta_\mathcal{Y}f(y)=\int_{\mathcal{X}} f(x,y)\ud\mu(x)-\beta\iint_{\mathcal{X}\times\mathcal{Y}}f(\chi,\upupsilon)\ud\mu(\chi)\ud\nu(\upupsilon),
$$
By Parseval in $L_2(\mu), L_2(\nu), L_2(\mu\times\nu)$,
\begin{align*}
\bigg\|S^\alpha_\mathcal{X}&f\bigg\|_{L_2(\mu)}^2+\left\|S^\beta_\mathcal{Y}f\right\|_{L_2(\nu)}^2\\&=\left(|1-\alpha|^2+|1-\beta|^2\right)\left|\langle f,\f_0\otimes \psi_0\rangle_{L_2(\mu\times \nu)}\right|^2
+\sum_{j=1}^\infty\left(\left|\langle f,\f_j\otimes \psi_0\rangle_{L_2(\mu\times \nu)}\right|^2+\left|\langle f,\f_0\otimes \psi_j\rangle_{L_2(\mu\times \nu)}\right|^2\right)\\&\stackrel{\eqref{eq:parseval use fg}}{\le} \max\left\{|1-\alpha|^2+|1-\beta|^2,1\right\}\|f\|_{L_2(\mu\times \nu)}^2.
\end{align*}
The case $\alpha=\beta=\frac12$ of this inequality is precisely~\eqref{eq:scalar 2}. It is worthwhile to note in passing that this reasoning (substituted into the above interpolation argument)  yields the following generalization of~\eqref{eq:two measures-2}.
\begin{align}\label{eq:two measures-2 alpha beta}
\begin{split}
\max\left\{\left(|1-\alpha|^2+|1-\beta|^2\right)^{p\theta},1\right\}&\left((1+|\alpha|)^{\frac{2p(1-\theta)}{2-\theta p}}+(1+|\beta|)^{\frac{2p(1-\theta)}{2-\theta p}}\right)^{1-\frac{\theta p}{2}}\iint_{\mathcal{X}\times \mathcal{Y}}\|f(x,y)\|_{\![F,H]_\theta}^p\ud\mu(x)\ud\nu(y)\\&\ge
\int_{\mathcal{X}}\bigg\|\int_\mathcal{Y} f(x,y)\ud\nu(y)-\alpha\iint_{\mathcal{X}\times\mathcal{Y}}f(\chi,\upupsilon)\ud\mu(\chi)\ud\nu(\upupsilon)\bigg\|_{\![F,H]_\theta}^p\ud\mu(x)
\\& \qquad+\int_{\mathcal{Y}}\bigg\|\int_\mathcal{X} f(x,y)\ud\mu(x)-\beta\iint_{\mathcal{X}\times\mathcal{Y}}f(\chi,\upupsilon)\ud\mu(\chi)\ud\nu(\upupsilon)\bigg\|_{\![F,H]_\theta}^p
\ud\nu(y).
\end{split}
\end{align}

For the justification of the remaining inequality~\eqref{eq:scalar 3}, define $Tg\in L_2(\mu)$ by
$$
Tg\eqdef \sum_{j=1}^\infty \left(\langle g,\f_0\otimes \f_j\rangle_{L_2(\mu\times \mu)}-\langle g,\f_j\otimes \f_0\rangle_{L_2(\mu\times \mu)}\right)\f_j.
$$
In other words, $\mu$-almost surely $Tg(\chi)=\int_{\mathcal{X}} \big(g(x,\chi)-g(\chi,x)\big)\ud\mu(x)$. By Parseval in $L_2(\mu),L_2(\mu\times\nu)$,
\begin{multline*}
\|Tg\|_{L_2(\mu)}^2=\sum_{j=1}^\infty \left|\langle g,\f_0\otimes \f_j\rangle_{L_2(\mu\times \mu)}-\langle g,\f_j\otimes \f_0\rangle_{L_2(\mu\times \mu)}\right|^2\\\le \sum_{j=1}^\infty 2\left(\left|\langle g,\f_0\otimes \f_j\rangle_{L_2(\mu\times \mu)}\right|^2+\left|\langle g,\f_j\otimes \f_0\rangle_{L_2(\mu\times \mu)}\right|^2\right)\stackrel{\eqref{eq:parseval use fg}}{\le}  2\|g\|_{L_2(\mu\times\mu)}^2,
\end{multline*}
where in the penultimate step we used the convexity of $(\zeta\in \C)\mapsto |\zeta|^2$. This is precisely~\eqref{eq:scalar 3}.
\end{proof}

We will next deduce Theorem~\ref{thm:Lp case} from the special case of Theorem~\ref{thm:interpolation} that we stated as Theorem~\ref{cor:differences}.

\begin{proof}[Proof of Theorem~\ref{thm:Lp case}] The largest $\theta\in [0,1]$  for which $\frac{2}{2-\theta}\le p\le \frac{2}{\theta}$ and also $\frac{1}{q}=\frac{1-\theta}{r}+\frac{\theta}{2}$ for some  $r\ge 1$ is
$$
\theta_{\max}=\theta_{\max}(p,q)\eqdef 2\min\left\{\frac{1}{p},1-\frac{1}{p},\frac{1}{q},1-\frac{1}{q}\right\}.
$$
We then have $L_q=[L_r,L_2]_{\theta_{\max}}$. Note that the quantity $c(p,q)$ that is defined in~\eqref{eq:df cpq} is equal to $\frac{p}{2}\theta_{\max}$.

By~\eqref{eq:rj interpolation} with $\theta=\theta_{\max}$ and $F=L_r$ we have $\boldsymbol{\mathcal{j}}_p(L_q)\ge 2^{c(p,q)}$. The matching upper bound $\boldsymbol{\mathcal{j}}_p(L_q)\le 2^{c(p,q)}$ holds due to  the following quick examples. If $X$ is uniformly distributed on $\{-1,1\}$, then $\E[|X-\E[X]\|^p]$ and $\E|X-X'|^p=2^{p-1}$. So, $\boldsymbol{\mathcal{j}}_p(\R)\le 2^{p-1}$. If $\e\in (0,1)$ and $\Pr[X_\e=0]=1-\e$ and $\Pr[X_\e=1]=\e$, then for  $p>1$, $$\boldsymbol{\mathcal{j}}_p(\R)\le \frac{\E\left[\|X_\e-X_\e'\|_q^p\right]}{\E\left[\|X_\e-\E[X_\e]\|_q^p\right]} = \frac{2\e(1-\e)}{(1-\e)\e^p+\e(1-\e)^p}\xrightarrow[\e\to 0^+]{}  2.$$
If $n\in \mathbb{N}$ and $X_n$ is uniformly distributed over  $\{\pm e_1,\ldots,\pm e_n\}$, where $\{e_j\}_{j=1}^\infty$ is the standard basis of $\ell_p$, then
$$
\boldsymbol{\mathcal{j}}_p(L_q)\le \boldsymbol{\mathcal{j}}_p(\ell_q^n)\le \frac{\E\left[\|X_n-X_n'\|_q^p\right]}{\E\left[\|X_n-\E[X_n]\|_q^p\right]}=\frac{n-1}{n}2^{\frac{p}{q}}+\frac{1}{2n}2^p\xrightarrow[n\to \infty]{} 2^{\frac{p}{q}}.
$$
If  $r_1,\ldots,r_n$ are i.i.d.~symmetric Bernoulli random variables viewed as elements of $L_q$, e.g.~they can be the coordinate functions in $L_q(\{-1,1\}^n)$, then let $R_n$ be uniformly distributed over $\{\pm r_1,\ldots,\pm r_n\}$. Then,
$$
\boldsymbol{\mathcal{j}}_p(L_q)\le  \frac{\E\left[\|R_n-R_n'\|_q^p\right]}{\E\left[\|R_n-\E[R_n]\|_q^p\right]}=\frac{n-1}{n}2^{\frac{p(q-1)}{q}}+\frac{1}{2n}2^p\xrightarrow[n\to \infty]{} 2^{\frac{p(q-1)}{q}}.
$$
This completes the proof that $\boldsymbol{\mathcal{j}}_p(L_q)=2^{c(p,q)}$.

Next, an application of~\eqref{eq:rj interpolation} with $\theta=\theta_{\max}$ and $F=L_r$ gives $\boldsymbol{\mathcal{r}}_p(L_q)\le 2^{1+(1-\theta_{\max})p}$. In other words,
\begin{equation}\label{eqînferior before nowflake}
\E\left[\|X-X'\|^p_{q}\right]+\E\left[\|Y-Y'\|^p_{q}\right]\le 2^{\max\left\{p-1,3-p,1+\frac{p(q-2)}{q},1+\frac{p(2-q)}{q}\right\}}\E\left[\|X-Y\|^p_{q}\right],
\end{equation}
for every $p$-integrable independent $L_q$-valued random variables $X,X',Y,Y'$ such that $(X,Y)$ and $(X',Y')$ are identically distributed.
The bound~\eqref{eqînferior before nowflake} coincides with~\eqref{eq:rj interpolation}, where $C(p,q)$ is as in~\eqref{eq:df Cpq}, only in the first two ranges that appear in~\eqref{eq:df Cpq}, namely when $\frac{p}{p-1}\le q\le p$ or when $\frac{q}{q-1}\le p\le q$. For the remaining ranges that appear in~\eqref{eq:df Cpq}, the bound~\eqref{eqînferior before nowflake} is inferior to~\eqref{eq:rj interpolation}, so we reason as follows.

For every $q,Q\in [1,\infty]$ satisfying $Q\ge q$, by~\cite[Remark~5.10]{MN04} (the case $Q\in [1,2]$  is an older result~\cite{BDK65}) there exists an embedding $\mathfrak{s}=\mathfrak{s}_{q,Q}:L_q\to L_Q$ (given by an explicit formula) such that
\begin{equation}\label{eq:snowflake quote}
\forall\, x,y\in L_q,\qquad \|\mathfrak{s}(x)-\mathfrak{s}(y)\|_Q=\|x-y\|_q^{\frac{q}{Q}}.
\end{equation}
Apply~\eqref{eqînferior before nowflake} to the $L_Q$-valued random vectors $\mathfrak{s}(X),  \mathfrak{s}(X'),\mathfrak{s}(Y),\mathfrak{s}(Y')$ with $q$ replaced by $Q$ and $p$ replaced with $\frac{pQ}{q}$. The resulting estimate is
\begin{align}\label{eq:change q Q}
\begin{split}
\E\left[\|X-X'\|^p_{q}\right]+\E\left[\|Y-Y'\|^p_{q}\right]&\stackrel{\eqref{eq:snowflake quote}}{=}
\E\left[\|\mathfrak{s}(X)-\mathfrak{s}(X')\|^{\frac{pQ}{q}}_{Q}\right]+\E\left[\|\mathfrak{s}(Y)-\mathfrak{s}(Y')\|^{\frac{pQ}{q}}_{Q}\right]\\
&\stackrel{\eqref{eqînferior before nowflake}}{\le} 2^{\max\left\{\frac{pQ}{q}-1,3-\frac{pQ}{q},1+\frac{p(Q-2)}{q},1+\frac{p(2-Q)}{q}\right\}}
\E\left[\|\mathfrak{s}(X)-\mathfrak{s}(Y)\|^{\frac{pQ}{q}}_{Q}\right]\\&\stackrel{\eqref{eq:snowflake quote}}{=} 2^{\max\left\{\frac{pQ}{q}-1,3-\frac{pQ}{q},1+\frac{p(Q-2)}{q},1+\frac{p(2-Q)}{q}\right\}}\E\left[\|X-Y\|^p_{q}\right].
\end{split}
\end{align}
It is in our interest to choose $Q\ge q$ so as to minimize the right hand side of~\eqref{eq:change q Q}. If $\frac{1}{p}+\frac{1}{q}\le 1$, then $Q=q$ is the optimal choice in~\eqref{eq:change q Q}, and therefore we return to~\eqref{eqînferior before nowflake}. But,  if $\frac{1}{p}+\frac{1}{q}\ge 1$, then $Q=1+\frac{q}{p}\ge q$ is the optimal choice in~\eqref{eq:change q Q} and we arrive at the following estimate which is better than~\eqref{eqînferior before nowflake} in the stated range
\begin{equation}\label{eq:1/p+1/q}
\frac{1}{p}+\frac{1}{q}\ge 1\implies \E\left[\|X-X'\|^p_{q}\right]+\E\left[\|Y-Y'\|^p_{q}\right]\le 2^{\max\left\{\frac{p}{q},2-\frac{p}{q}\right\}}\E\left[\|X-Y\|^p_{q}\right].
\end{equation}
The bound~\eqref{eq:1/p+1/q} covers the third and fourth ranges that appear in~\eqref{eq:df Cpq}, as well as the case $p=q\in [1,2]$ of the fifth range that appears in~\eqref{eq:df Cpq}. However, \eqref{eq:1/p+1/q} is inferior to~\eqref{eq:df Cpq} when $1\le p<q\le 2$. When this occurs, use the fact~\cite{Kad58} that $L_q$ is isometric to a subspace of $L_p$ and apply the already established case $p=q$ to the $L_p$-valued random variables $\mathcal{i}(X), \mathcal{i}(X'),\mathcal{i}(Y),\mathcal{i}(Y')$, where $\mathcal{i}:L_q\to L_p$ is any isometric embedding.

We will next  prove that $\boldsymbol{\mathcal{r}}_p(L_q)\ge 2^{C_{\mathrm{opt}}(p,q)}$, where $C_{\mathrm{opt}}(p,q)$ is given in~\eqref{eq:df cpq opt}. In particular, this will justify the second sharpness assertion of Theorem~\ref{thm:Lp case}, namely that~\eqref{eqînferior before nowflake} is sharp when $p,q$ belong to the first, second or fifth ranges that appear in~\eqref{eq:df Cpq}. Firstly, by considering the special case of~\eqref{eq:roundness} in which $X,Y$ are i.i.d., we see that $\boldsymbol{\mathcal{r}}_p(F)\ge 1$ for any Banach space $F$. Next, fix $n\in \mathbb{N}$ and let $r_1,\ldots,r_n,\rho_1,\ldots,\rho_n\in L_q$ be such that $r_1,\ldots,r_n$ and $\rho_1,\ldots,\rho_n$  each form a sequence of i.i.d.~symmetric Bernoulli random variables, and the supports of $r_1,\ldots,r_n$ are disjoint from the supports of $\rho_1,\ldots,\rho_n$. For example, one could consider them as the elements of $L_q(\{-1,1\}^{n})\oplus_q L_q(\{-1,1\}^{n})$ that are given by $r_i=(\omega\mapsto\omega_i,0)$ and $\rho_i=(0,\omega\mapsto\omega_i)$ for each $i\in \n$. Let $X$ be uniformly distributed over $\{r_1,\ldots,r_n\}$ and $Y$ be uniformly distributed over $\{\rho_1,\ldots,\rho_n\}$. Due to the disjointness of the supports, we have  $\|X-Y\|^p_q=(\|X\|_q^q+\|Y\|_q^q)^{p/q}=2^{p/q}$ point-wise. At the same time,  $\E[\|X-X'\|_q^p]+\E[\|Y-Y'\|_q^p]=2(1-1/n)(2^q/2)^{p/q}=(1-1/n)2^{1+p(q-1)/q}$. By letting $n\to\infty$, this shows that necessarily $\boldsymbol{\mathcal{r}}_p(L_q)\le 2^{1+p(q-2)/q}$. Finally, if~\eqref{eq:roundness} holds, then in particular it holds for scalar-valued random variables. By integrating,  we see that $\boldsymbol{\mathcal{r}}_p(F)\ge \boldsymbol{\mathcal{r}}_p(L_p)$ for any Banach space $F$. But, the case $p=q$ of the above discussion gives  $\boldsymbol{\mathcal{r}}_p(L_p)\ge 2^{1+p(p-2)/p}=2^{p-1}$, as required.

The bound~\eqref{eq:bm bound} of Theorem~\ref{thm:Lp case} coincides with~\eqref{eq:mb interpolation}. When $p \leq q \leq 2$, we have $C(p,q)=1$, $c(p,q)=p-1$ and thus $\boldsymbol{\mathcal{m}}_p(L_q) \leq 2^{2-p}$. It therefore remains to check that $\boldsymbol{\mathcal{b}}_p(L_q)\ge 2^{2-p}$ when $p\le q\le 2$. In fact,  $\boldsymbol{\mathcal{b}}_p(F)\ge 2^{2-p}$ for every $p\ge 1$ and every Banach space $(F,\|\cdot\|_F)$. Indeed, fix distinct $a,b\in F$. Let $X,Y$ be independent and uniformly distributed over $\{a,b\}$. Then
\begin{equation*}
\boldsymbol{\mathcal{b}}_p(F)\ge \frac{\E\left[\|X-z\|_{\!F}^p+\|Y-z\|_{\!F}^p\right]}{\E\left[\|X-Y\|_{\!F}^p\right]}=\frac{\|a-z\|_{\!F}^p+\|b-z\|_{\!F}^p}{\frac12\|a-b\|_{\!F}^p}\ge \frac{2\left\|\frac12(a-z)-\frac12(b-z)\right\|_{\!F}^p}{\frac12\|a-b\|_{\!F}^p}=2^{2-p},
\end{equation*}
where the penultimate step is an application of the convexity of $\|\cdot\|_F^p$.
\end{proof}
\begin{remark}\label{rem:schatten}
Fix $n\in \mathbb{N}$. Following~\cite{BRS17}, for $a=(a_1,\ldots,a_{2n})\in \C^{2n}$ denote by $\Re(a)=(\Re(a_1),\ldots,\Re(a_{2n}))\in \R^{2n}$ and $\Im(a)=(\Im(a_1),\ldots,\Im(a_{2n}))\in \R^{2n}$ the vectors of real parts and imaginary parts of the entries of $a$, respectively. Let $\Lambda(a)\in [0,\infty)$ be the area of the parallelogram that is generated by $\Re(a)$ and $\Im(a)$, i.e.,
$$
\Lambda(a)\eqdef \sqrt{\|\Re(a)\|_2^2\|\Im(a)\|_2^2-\left\langle\Re(a),\Im(a)\right\rangle }.
$$
By~\cite[Lemma~5.2]{BRS17} there is a linear operator $\mathcal{C}:\C^{2n}\to \mathsf{M}_{2^n}(\C)$ from $\C^{2n}$ to the space of $2^n$ by  $2^n$ complex matrices, such that for any $a\in \C^{2n}$ the Schatten-1 norm of the matrix $\mathcal{C}(a)$ satisfies
\begin{equation}\label{eq:Lambda formula}
\|\mathcal{C}(a)\|_{\mathsf{S_1}}=\frac12\sqrt{\|a\|_2^2+2\Lambda(a)}+\frac12\sqrt{\|a\|_2^2-2\Lambda(a)}.
\end{equation}
Let $e_1,\ldots,e_{2n}\in \C^{2n}$ be the standard  basis of $\C^{2n}$ and define $2n$ matrices $x_1,\ldots,x_n,y_1,\ldots,y_n\in \mathsf{M}_{2^n}(\C)$ by   $x_k=\mathcal{C}(e_k)$ and $y_k=\mathcal{C}(ie_{n+k})$ for  $k\in \n$. By~\eqref{eq:Lambda formula} we have $\|x_j-x_k\|_{\mathsf{S}_1}=\|y_j-y_k\|_{\mathsf{S}_1}=\sqrt{2}$ for distinct $j,k\in \n$, while $\|x_j-y_k\|_{\mathsf{S}_1}=1$ for all $j,k\in \n$. Hence, if we let $X$ and $Y$ be independent and distributed uniformly over $\{x_1,\ldots,x_n\}$ and $\{y_1,\ldots,y_n\}$, respectively, and $X',Y'$ are independent copies of $X,Y$, respectively, then for every $p\ge 1$ we have
$$\E\left[\|X-Y\|_{\mathsf{S}_1}^p\right]=1\qquad\mathrm{and}\qquad \E\left[\|X-X'
\|_{\mathsf{S}_1}^p\right]=\E\left[\|Y-Y'\|_{\mathsf{S}_1}^p\right]=\frac{n-1}{n}2^{\frac{p}{2}}.$$
By letting $n\to\infty$, this implies that $\boldsymbol{\mathcal{r}}_p(\mathsf{S}_1)\ge 2^{\frac{p}{2}+1}$. In particular, $\boldsymbol{\mathcal{r}}_1(\mathsf{S}_1)\ge 2\sqrt{2}$.
\end{remark}

\begin{remark} \label{infsetting} Fix $q\ge 1$.  Let $(F, \| \cdot\|_{\!F})$ be a Banach space.
Assume that $F$ has a linear subspace $G\subset F$ that is isometric to $L_q$ (or the Schatten--von Neumann trace class $\mathsf{S}_q$).
If $X,Y\in L_p(G)$ are i.i.d.~random variables taking values in  $G$, then for $c(p,q)$ as in~\eqref{eq:df cpq}, by Theorem~\ref{thm:Lp case} we have
\begin{equation}\label{eq:inf z}
\E\left[\|X-Y\|_{\!F}^p\right] \ge 2^{c(p,q)}\cdot \inf_{z \in F} \E \left[\|X-z\|_{\! F}^p\right].
\end{equation}
We note that this inequality is optimal despite the fact that the infimum is now taken over $z$ in the larger super-space $F$. Indeed, in the proof of Theorem~\ref{thm:Lp case} the random variables that established optimality of $c(p,q)$ were symmetric when $p,q$ belong to the first three ranges that appear in~\eqref{eq:df cpq}. In these cases, by the convexity of $\| \cdot\|_{\!F}^{p}$, the infimum in the right had side of~\eqref{eq:inf z} is attained at $z=0\in G$. The fact that the term $2^{c(p,q)}$ in the right hand side of~\eqref{eq:inf z} cannot be replaced by any value greater than $2$  needs the following separate treatment. If $\e\in (0,1)$ and  $\Pr[X=v]=\varepsilon=1-\Pr[X=0]$ for some $v \in G$ with $\|v\|_{\!F}=1$, then $\E\left[\|X-Y\|_{\!F}^{p}\right]=2\e(1-\e)$. Next, for any $z \in F$ we have
\begin{multline*}
\E\left[\|X-z\|_{\!F}^{p}\right]=(1-\e)\|z\|_{\!F}^p+\e\|v-z\|_{\!F}^p\ge (1-\e)\|z\|_{\!F}^p+\e \left(\max\left\{0,1-\|z\|_{\!F}^{\phantom{p}}\right\}\right)^p\\\ge \min_{r\ge 0} \left((1-\e)r^p+\e \left(\max\left\{0,1-r\right\}\right)^p\right)=\frac{\e(1-\e)}{\left(\e^{\frac{1}{p-1}}+(1-\e)^{\frac{1}{p-1}}\right)^{p-1}},
\end{multline*}
where the final step follows by elementary calculus. Therefore,
$$
\frac{\E\left[\|X-Y\|_{\!F}^p\right]}{  \inf_{z \in F} \E \left[\|X-z\|_{\! F}^p\right]}\le 2\left(\e^{\frac{1}{p-1}}+(1-\e)^{\frac{1}{p-1}}\right)^{p-1}\xrightarrow[\e\to 0^+]{}  2.
$$
\end{remark}

\begin{remark}
An extrapolation theorem of Pisier~\cite{Pis79} asserts that if $(F,\|\cdot\|_{\!F})$ is a Banach lattice that is both $p$-convex with constant $1$ and $q$-concave with constant $1$, where $\frac{1}{p}+\frac{1}{q}=1$, then there exists a Banach lattice $W$, a Hilbert space $H$, and $\theta\in (0,1]$ such that $F$ is isometric to the complex interpolation space $[W,H]_\theta$. Hence, Theorem~\ref{cor:differences} applies in this setting, implying in particular that there is $r\in [1,\infty)$, namely $r=\frac{2}{\theta}$, such that every i.i.d.~$F$-valued random variables $X,Y\in L_r(F)$ satisfy
$$
\E\left[\|X-Y\|_{\!F}^r\right]\ge 2\E\left[\big\|X-\E[X]\big\|_{\! F}^r\right].
$$
\end{remark}

We will  conclude by discussing further bounds  in the non-convex range $p<1$, as well as their limit when $p\to 0^+$. When $p\in (0,1)$, the topological vector space $L_p$ is not a normed space. Despite this, when we say that a normed space $(F,\|\cdot\|_{\!F})$ admits a linear isometric emebdding into $L_p$ we mean (as usual) that there exists a linear mapping $T:\to L_p$ such that $\|Tx\|_p=\|x\|_{\!F}$ for all $x\in F$. This of course forces the $L_p$ quasi-norm to induce a metric on the image of $T$, so the use of the term ``iso{\bf metric}'' is not out of place here, though note that it is inconsistent with the standard metric on $L_p$, which is given by $\|f-g\|_p^p$ for all $f,g\in L_p$. The following proposition treats the case $p\in (0,2]$, though later we will mainly be interested in the non-convex range $p\in (0,1)$. Note that the case $p=1$ implies the stated inequalities for, say, any two-dimensional normed space, since any such space admits~\cite{Bol69} an isometric embedding into $L_1$.

\begin{proposition} \label{embedding}
Let $(F, \| \cdot \|_{\! F})$ be a Banach space  that admits
an isometric linear embedding into $L_{p}$ for some $p \in (0,2]$. Let $X,X',Y,Y'\in L_p(F)$ be independent $F$-valued random vectors such that $X'$ has the same distribution as $X$ and $Y'$ has the same distribution as $Y$.
Then,
\begin{equation}\label{eq:nonconvex roundness}
\E\left[\left\|X-X'\right\|_{\!F}^p\right]+\E\left[\left\|Y-Y'\right\|_{\!F}^p\right]\le 2\E\left[\left\|X-Y\right\|_{\!F}^p\right],
\end{equation}
and
\begin{equation}\label{eq:nonconvex bary}
\inf_{z \in F} \E\left[\|X-z\|_{\!F}^{p}+\E\|Y-z\|_{\!F}^{p}\right]
\leq \min\left\{2,2^{2-p}\right\}  \E\left[\|X-Y\|_{\!F}^{p}\right].
\end{equation}
The constants $2$ and $\min\left\{2,2^{2-p}\right\}$ in~\eqref{eq:nonconvex roundness} and~\eqref{eq:nonconvex bary}, respectively, cannot be improved.
\end{proposition}

\begin{proof}
By~\cite{Sch38,BDK65} there is a mapping $\mathfrak{s}:F\to L_2$ such that $\|\mathfrak{s}(x)-\mathfrak{s}(y)\|_2=\|x-y\|_{\!F}^{\!\frac{p}{2}}$ for all $x,y\in F$.  By the (trivial) Hilbertian case $p=q=2$ of Theorem~\ref{thm:Lp case} applied to the $L_2$-valued random vectors $\mathfrak{s}(X), \mathfrak{s}(Y)$,
\begin{align*}
\E\left[\left\|X-X'\right\|_{\!F}^p\right]+\E\left[\left\|Y-Y'\right\|_{\!F}^p\right]&=
\E\left[\left\|\mathfrak{s}(X)-\mathfrak{s}(X')\right\|_{2}^2\right]+\E\left[\left\|\mathfrak{s}(Y)-\mathfrak{s}(Y')\right\|_{2}^2\right]\\&\le 2\E\left[\left\|\mathfrak{s}(X)-\mathfrak{s}(Y)\right\|_{2}^2\right]=2\E\left[\left\|X-Y\right\|_{\!F}^p\right].
\end{align*}
This substantiates~\eqref{eq:nonconvex roundness}. When $p<1$ we cannot proceed from here to prove~\eqref{eq:nonconvex bary} by considering the analogue of the mixture constant $\boldsymbol{\mathcal{m}}(\cdot)$, namely by bounding the left hand side of~\eqref{eq:def m} as we did in the Introduction, since the present $L_p$ integrability assumption on $X,Y$ does not imply that $\E[X]$ and $\E[Y]$ are well-defined elements of $F$. Instead, let $Z'$ be independent of $X,Y$ and distributed according to the mixture of the laws of $X$ and $Y$, as in~\eqref{eq:def mixture}. The point $z\in F$ will be chosen randomly according to $Z'$, i.e.,
\begin{align}\label{eq:pass from Z to b}
\begin{split}
\inf_{z\in F}\E\left[\left\|X-z\right\|_{\!F}^p+\left\|Y-z\right\|_{\!F}^p\right]&\le \E\left[\left\|X-Z'\right\|_{\!F}^p+\left\|Y-Z'\right\|_{\!F}^p\right]
\\&=\frac12\E\left[\left\|X-X'\right\|_{\!F}^p+\left\|Y-Y'\right\|_{\!F}^p\right]+\E\left[\left\|X-Y\right\|_{\!F}^p\right]\stackrel{\eqref{eq:nonconvex roundness}}{\le} 2\E\left[\left\|X-Y\right\|_{\!F}^p\right].
\end{split}
\end{align}
For $p\ge 1$ we have $\boldsymbol{\mathcal{r}}_p(F)\le 2$ by~\eqref{eq:nonconvex roundness}, and  $\boldsymbol{\mathcal{j}}_p(F)\ge \boldsymbol{\mathcal{j}}_p(L_p)=p-1$ by Theorem~\ref{thm:Lp case}, so $\boldsymbol{\mathcal{b}}_p(F)\le 2^{2-p}$, by~\eqref{eq:general moduli relations}.

The sharpness of~\eqref{eq:nonconvex roundness} is seen by taking $X$ and $Y$ to be identically distributed. When $p\ge 1$, we already saw in the proof of Theorem~\ref{thm:Lp case} that $\boldsymbol{\mathcal{b}}_p(F)\ge 2^{2-p}$ for any Banach space $F$; thus~\eqref{eq:nonconvex bary} is sharp in this range. The same reasoning as in the proof of Theorem~\ref{thm:Lp case}  shows that the factor $2$ in~\eqref{eq:nonconvex bary} cannot be improved in the non-convex range $p\in (0,1)$ as well. Indeed, fix $v$ with $\|v\|_{\!F}=1$ and let $X$ and $Y$ be uniformly distributed over $\{0,v\}$.  Then, $\E\left[\|X-z\|_{\!F}^{p}+\E\|Y-z\|_{\!F}^{p}\right]=\|z\|_{\!F}^p+\|v-z\|_{\!F}^p\ge (\|z\|_{\!F}+\|v-z\|_{\!F})^p\ge \|v\|_{\!F}^p=1$ for every $z\in F$, while $\E\left[\|X-Y\|_{\!F}^{p}\right]=\frac12 \|v\|_{\! F}^p=\frac12$.
\end{proof}

Proposition~\ref{0-mean} below is the  limit of Proposition~\ref{embedding}  as $p\to 0^+$. While it is possible to deduce it formally from  Proposition~\ref{embedding} by passing to the limit, a justification of this  fact is quite complicated due to the singularity of the logarithm at zero. We will instead proceed via a shorter alternative approach.

Following~\cite{KKYY07}, a real Banach space $(F,\|\cdot\|_{\!F})$ is said to admit a linear isometric embedding into $L_0$ if there exists a probability space $(\Omega,\mu)$ and a linear operator $T:F\to \mathsf{Meas}(\Omega,\mu)$, where $\mathsf{Meas}(\Omega,\mu)$ denotes the space of (equivalence classes of) real-valued $\mu$-measurable functions on $\Omega$, such that
\begin{equation}\label{eq:exponential}
\forall\, x\in F,\qquad \|x\|_{\!F}=e^{\int_{\Omega}\log |Tx|\ud\mu}.
\end{equation}
As shown in~\cite{KKYY07}, every three-dimensional real normed space admits a linear isometric embedding into $L_0$, so in particular the following proposition applies to any such space.

\begin{proposition} \label{0-mean}
Let $(F,\| \cdot\|_{\!F})$ be a real Banach space that
admits a linear isometric embedding into $L_{0}$. Let $X,X',Y,Y'$ be independent $F$-valued random vectors such that $X'$ has the same distribution as $X$ and $Y'$ has the same distribution as $Y$. Assume that $\E\left[\log(1+\|X\|_{\!F})\right]<\infty$ and $\E\left[\log(1+\|Y\|_{\!F})\right]<\infty$. Then,
\begin{equation}\label{eq:product roundness}
e^{\E\left[\log\left(\|X-X'\|_{\!F}\cdot \|Y-Y'\|_{\!F}\right)\right]}\le e^{2\E\left[\log\left(\|X-Y\|_{\!F}\right)\right]},
\end{equation}
and
\begin{equation}\label{eq:product bary}
\inf_{z\in F}e^{\E\left[\log \left(\|X-z\|_{\!F}\cdot\|Y-z\|_{\!F}\right)\right]}\le e^{2\E\left[\log\left(\|X-Y\|_{\!F}\right)\right]}.
\end{equation}
The multiplicative constant $1$ in both of these inequalities is optimal.
\end{proposition}

\begin{proof} \eqref{eq:product bary} is a consequence of~\eqref{eq:product roundness}  by reasoning analogously to~\eqref{eq:pass from Z to b}.  Due to the assumed representation~\eqref{eq:exponential}, by Fubini's theorem  it suffices to prove~\eqref{eq:product roundness} for real-valued random variables.

So, suppose that $X,Y$ are independent real-valued random variables such that $\E\left[\log(1+|X|)\right]<\infty$ and $\E\left[\log(1+|Y|)\right]<\infty$. Note that  every nonnegative random variable $W$ with $\E\left[\log(1+W)\right]<\infty$  satisfies
\begin{equation}\label{eq:integral laplace}
\E\left[\log W\right] =\int_0^\infty \frac{e^{-s}-\E\left[e^{-sW}\right]}{s}\ud s.
\end{equation}
Indeed, for every $a,b\in [0,\infty)$ with $a\le b$ we have
$$
\int_0^\infty \frac{e^{-as}-e^{-bs}}{s}\ud s=\int_{0}^{\infty} \left( \int_{a}^{b} se^{-ts}\,\ud t \right)\frac{\ud s}{s}=
\int_{a}^{b} \left( \int_{0}^{\infty} e^{-ts}\,\ud s\right)\,\ud t=\int_{a}^{b} \frac{\ud t}{t}=\log b-\log a,
$$
so that~\eqref{eq:integral laplace} follows by applying this identity and  the Fubini theorem separately on each of the events
$\{ W \geq 1\}$ and $\{ W<1\}$, taking advantage of the fact that $e^{-s}-e^{-sW}$ is of constant sign on both events.

Let $Z,Z'$ be independent random variables whose law is the mixture of the laws of $X,Y$ as in~\eqref{eq:def mixture}. the desired inequality~\eqref{eq:product roundness} is equivalent to the assertion that $\E\left[\log (Z-Z')^2\right]\le \E\left[\log (X-Y)^2\right]$. By two applications of~\eqref{eq:integral laplace}, once with $W=(X-Y)^2$ and once with $W=(Z-Z')^2$, it suffices to prove that
$$
\forall\, s\ge 0,\qquad \E\left[e^{-s(Z-Z')^{2}}\right] \geq \E\left[e^{-s(X-Y)^{2}}\right].
$$
This is so because, using the formula for the Fourier transform of the Gaussian density, we have
\begin{align}
\nonumber \E\left[e^{-s(Z-Z')^{2}}\right] &=\E\left[\frac{1}{\sqrt{2\pi}}\int_{-\infty}^\infty e^{it(Z-Z')\sqrt{2s}-\frac{t^2}{2}} \ud t\right]\\&=\frac{1}{\sqrt{2\pi}}\int_{-\infty}^\infty \E\left[e^{it\sqrt{2s} Z}\right]\cdot \E\left[e^{-it\sqrt{2s} Z'}\right] e^{-\frac{t^2}{2}} \ud t\nonumber \\&=\frac{1}{\sqrt{2\pi}}\int_{-\infty}^\infty \left|\E\left[e^{it\sqrt{2s} Z}\right]\right|^2e^{-\frac{t^2}{2}} \ud t\label{eq:use ZZ' ind}\\
&= \frac{1}{\sqrt{2\pi}}\int_{-\infty}^\infty \left|\frac12\E\left[e^{it\sqrt{2s} X}\right]+\frac12 \E\left[e^{it\sqrt{2s} Y}\right]\right|^2e^{-\frac{t^2}{2}} \ud t \nonumber\\
&\ge \frac{1}{\sqrt{2\pi}}\int_{-\infty}^\infty \Re\left(\E\left[e^{it\sqrt{2s} X}\right]\cdot \overline{\E\left[e^{it\sqrt{2s} Y}\right]}\right)e^{-\frac{t^2}{2}} \ud t \label{eq:amgm CIOMPLEX}\\
&= \frac{1}{\sqrt{2\pi}}\Re\left(\int_{-\infty}^\infty \E\left[e^{it\sqrt{2s} (X-Y)}\right]e^{-\frac{t^2}{2}} \ud t\right)=\E\left[e^{-s(X-Y)^{2}}\right]\label{eq:ind XY},
\end{align}
where~\eqref{eq:use ZZ' ind} uses Fubini and the independence of $Z$ and $Z'$,  \eqref{eq:amgm CIOMPLEX} uses the fact that for all $a,b\in \C$ we have $|(a+b)/2|^2=|(a-b)/2|^2+\Re(a\overline{b})\ge \Re(a\overline{b})$, the first step of~\eqref{eq:ind XY} uses the independence of $X$ and $Y$, and the last step of~\eqref{eq:ind XY} uses once more the formula for the Fourier transform of the Gaussian density.

The fact that~\eqref{eq:product roundness} is sharp follows by considering the case when $X,Y$ are i.i.d.~and non-atomic. Note that when both $X$ and $Y$ have an atom at the same point, both sides of~\eqref{eq:product bary} equal $0$. The example considered in the proof of Proposition~\ref{embedding} when $p>0$ is therefore of no use for establishing the optimality of~\eqref{eq:product bary}, due to the atomic nature of the distributions under consideration.
Instead, for an arbitrary $v \in F$ such that $\| v\|_{\!F}=1$, let us consider random vectors $X=(\cos \Theta) v$ and $Y=(\cos \Theta') v$, where $\Theta$ and $\Theta'$ are independent random variables uniformly distributed on $[0,2\pi]$.

Observe that for every $\alpha\in \R$ we have
\begin{align}\label{eq:Theta alpha}
\begin{split}
\E\left[\log \left|\cos\Theta-\cos\alpha\right|\right]&=\E\left[\log \left|2\sin\left(\frac{\Theta+\alpha}{2}\right)\sin\left(\frac{\Theta-\alpha}{2}\right)\right|\right]\\&=
\log 2+\E\left[\log \left|\sin\left(\frac{\Theta+\alpha}{2}\right)\right|\right]+\E\left[\log\left|\sin\left(\frac{\Theta-\alpha}{2}\right)\right|\right]=\log 2+2\E\left[\log \left|\cos\Theta\right|\right],
\end{split}
\end{align}
where the last step of~\eqref{eq:Theta alpha} holds because, by periodicity, $\left|\sin\left(\frac{\Theta\pm \alpha}{2}\right)\right|$ has the same distribution as $\left|\cos\Theta\right|$.

The case $\alpha=\frac{\pi}{2}$ of~\eqref{eq:Theta alpha} simplifies to give $\E\left[\log \left|\cos\Theta\right|\right]=-\log 2$. Hence, \eqref{eq:Theta alpha} becomes
\begin{equation}\label{eq:-log2}
\forall\, \alpha\in \R,\qquad \E\left[\log \left|\cos\Theta-\cos\alpha\right|\right]=-\log 2.
\end{equation}
Consequently,
\begin{equation}\label{eq: at leas -log 2}
\forall\, t\in \R,\qquad \E\left[\log \left|\cos\Theta-t\right|\right]\ge-\log 2.
\end{equation}
Indeed, if $t\in [-1,1]$, then one can write $t=\cos\alpha$ for some $\alpha\in \R$, so that by~\eqref{eq:-log2} the inequality in~\eqref{eq: at leas -log 2} holds as equality. If $|t|>1$, then $\left|\cos\theta-t\right|\ge \left|\cos\theta-\mathrm{sign}(t)\right|$ for all $\theta\in [0,2\pi]$, thus implying~\eqref{eq: at leas -log 2}. It also follows from~\eqref{eq:-log2} that
$$
\E\left[\log\left(\|X-Y\|_{\!F}\right)\right]=\E\left[\log \left|\cos\Theta-\cos\Theta'\right|\right]\stackrel{\eqref{eq:-log2}}{=}-\log 2.
$$
Next, by the Hahn--Banach theorem, take $\varphi \in F^{*}$ such that $\| \varphi\|_{F^{*}}=1$ and $\varphi(v)=\| v\|_{\! F}=1$. For any $z \in F$,
$$
\E\left[\log \left(\|X-z\|_{\!F}\right)\right] =\E\left[\log\left(\|Y-z\|_{\!F}\right)\right]\ge \E\left[\log \left|\varphi((\cos\Theta)v-z)\right|\right]=\E\left[\log \left|\cos\Theta-\varphi(z)\right|\right]\stackrel{\eqref{eq: at leas -log 2}}\ge -\log 2.
$$

This implies the asserted sharpness of~\eqref{eq:product bary}. Note that the above argument that~\eqref{eq:product bary} cannot hold with a multiplicative constant less than $1$ in the right hand side worked for any  Banach space $F$ whatsoever.
\end{proof}

\subsection*{Acknowledgements} We are grateful to Oded Regev for pointing us to~\cite[Lemma~5.2]{BRS17} and for significantly simplifying our initial reasoning  for the statement that is proved in Remark~\ref{rem:schatten}.
\bibliographystyle{abbrv}
\bibliography{no}

\bigskip

\noindent
{\sc Assaf Naor}\\
Department of Mathematics, Princeton University, Fine Hall Washington Road, Princeton NJ, USA;\\
e-mail: naor@math.princeton.edu

\bigskip
\noindent
{\sc Krzysztof Oleszkiewicz}\\
Institute of Mathematics, University of Warsaw, ul. Banacha 2, Warszawa, Poland;\\
Institute of Mathematics, Polish Academy of Sciences, ul. \'Sniadeckich 8, Warszawa, Poland;\\
e-mail: koles@mimuw.edu.pl

\end{document}